\documentclass[12pt,letterpaper,leqno]{amsart}

\usepackage[latin1]{inputenc}
\usepackage[T1]{fontenc}
\usepackage{amsfonts}
\usepackage{amsmath}
\usepackage{amssymb}
\usepackage{eurosym}
\usepackage{mathrsfs}
\usepackage{palatino}

\newcommand{\R}{\mathbb{R}}
\newcommand{\C}{\mathbb{C}}

\newcommand{\N}{\mathbb{N}}

\newcommand{\Z}{\mathbb{Z}}

\numberwithin{equation}{section}

\swapnumbers
\theoremstyle{plain}
\newtheorem{thm}[equation]{Theorem}
\newtheorem{lem}[equation]{Lemma}
\newtheorem{prop}[equation]{Proposition}

\theoremstyle{definition}
\newtheorem{defn}[equation]{Definition}

\theoremstyle{remark}

\newcommand{\pair}[2]{\langle #1,#2 \rangle}
\newcommand{\ave}[1]{\langle #1 \rangle}

\newcommand{\good}[0]{\operatorname{good}}

\newcommand{\children}[0]{\operatorname{ch}}
\newcommand{\out}[0]{\operatorname{out}}
\newcommand{\inside}[0]{\operatorname{in}}
\newcommand{\near}[0]{\operatorname{near}}
\newcommand{\dist}[0]{\operatorname{dist}}

\title{A $T(b)$ Theorem on Product Spaces}

\author{Yumeng Ou}\thanks{The author is partially supported by NSF-DMS 0901139 and ARC DP 120100399.}
\address{Department of Mathematics, Brown University, Providence, RI, USA}
\email{yumeng\_ou@brown.edu}


\begin{document}
\maketitle

\begin{abstract}
The main result of this paper is a bi-parameter $T(b)$ theorem for the case that $b$ is a tensor product of two pseudo-accretive functions. In the proof, we also discuss the $L^2$ boundedness of different types of the $b$-adapted bi-parameter paraproducts.
\end{abstract}

\section{Introduction}

The study of the $T(1)$/$T(b)$ type theorems in the multi-parameter setting can be dated back to 1985, when Journ\'e \cite{Jo} proved the first multi-parameter $T(1)$ theorem by treating the singular integral operator as a vector-valued one-parameter operator. The result itself is very elegant except that some partial boundedness of the operator needs to be assumed. More recently, Pott and Villarroya in \cite{PV} prove a new bi-parameter $T(1)$ theorem with much weaker assumptions on the operator, where they formed different types of mixed conditions instead of assuming the partial boundedness. This is the point of view taken by Martikainen in \cite{Ma}, where he proved a representation theorem for bi-parameter singular integral operators which then implies a $T(1)$ result, and in his joint work with Hyt\"onen \cite{HM2}, where they showed a bi-parameter $T(1)$ theorem in spaces of non-homogeneous type.\footnote{The paper \cite{Ma} and \cite{HM2} cite a 2011 version of \cite{PV} which was revised in February 2013. See Section 2 for a further discussion.}

In this paper, for the first time, we prove a $T(b)$ theorem in product spaces, which is a natural extension of the work we have mentioned above. 

\begin{defn}
A function $b\in L^{\infty}(\R^n\times\R^m)$ is called pseudo-accretive if there is a constant $C$ such that for any rectangle $R$ in $\R^n\times\R^m$ with sides parallel to axes, $\frac{1}{|R|}|\int_{R}b|>C$.
\end{defn}

We will only discuss the case when $b=b_1\otimes b_2$, where $b_1$ and $b_2$ are in $L^\infty(\R^n)$ and $L^\infty(\R^m)$, respectively. Then,  the pseudo-accretivity and boundedness of $b$ imply that there exists a constant $C$, such that for any cubes $K\subset \R^n, V\subset\R^m$, $\frac{1}{|K|}|\int_{K}b_1|>C$ and $\frac{1}{|V|}|\int_{V}b_2|>C$, i.e. $b_1$ and $b_2$ are both pseudo-accretive in the classical sense. Although this seems to be too restrictive, it is actually quite natural. Note that $b=1$ falls in this class. Moreover, in all of the papers mentioned above, some partial structures on the operator are required in order to treat those mixed problems risen in the bi-parameter setting. In other words, the singular integral operator itself we are looking at behaves like a tensor product in some sense. It is essential in our argument for $b$ to be a tensor product, otherwise, even defining $Tb$ would become a problem.

Just as in the situation for the bi-parameter $T(1)$ theorems, we still need to assume that besides $T, T^*$, the partial adjoints of $T$ also map $b$ to a $BMO$ function, an assumption shown by Journ\'e \cite{Jo} to be unnecessary for $T$ to be $L^2$ bounded. A more detailed discussion can be found in Section 6 of \cite{Jo}.

The main technique of the proof is to decompose $L^2$ functions into sums of martingale differences adapted to $b$, analyze each part of the sums, and show that they have good enough decay to be summed up. The advantage of analyzing martingale differences is that they are supported on dyadic rectangles, constant on each of their children, and have orthogonality. Martikainen followed a similar strategy in \cite{Ma}, using Haar functions. However, when we treat $b$ instead of $1$, we have to create a bi-parameter $b$-adapted martingale difference decomposition, which makes the estimate of each part of the sum much less transparent. In the one-parameter setting, the idea of using such $b$-adapted martingale difference operators is well known and has been discussed by many authors in their proofs of different types of $Tb$ theorems, such as David, Journ\'e and Semmes \cite{DJS}, Coifman, Jones and Semmes \cite{CJS}, Nazarov, Treil and Volberg \cite{NTV}, Hyt\"onen and Martikainen \cite{HM}. But in the bi-parameter case, the $b$-adapted martingale difference has never been treated before.

The operator $T$ studied in this paper is initially defined as a continuous linear map from $bC^\infty_0(\R^n\times\R^m)$ to its dual. In order to justify the convergence of pairings of martingale differences, we also assume a priori that $T$ is bounded on $L^2(\R^n\times\R^m)$, although we will show that quantitatively the operator norm of $T$ is bounded by some constant depending only on the weak assumptions introduced in the following, but has nothing to do with the assumed $L^2\rightarrow L^2$ norm. Note that although this a priori assumption is often unnecessary, it appears as a hypothesis in the proofs of some $T(1)$ theorems: many authors have added this assumption (\cite{Ma}, \cite{HM2}), even in the one-parameter setting (\cite{NTV}, \cite{Hy}). It is not a consequence of involving $b$, but results from the fact that one has an initially continuously defined operator which is treated dyadically. Thus, we are more interested in showing how those weak assumptions quantitatively control the $L^2\rightarrow L^2$ norm of $T$. However, in some specific examples that we will mention later, this a priori assumption can be removed.

The plan for the paper is the following. First, we introduce the assumptions on the operators as well as necessary preliminary on bi-parameter $b$-adapted martingale differences. Second, before stating and proving the $T(b)$ theorem, we discuss some types of bi-parameter $b$-adapted paraproducts, which will be used later. Next, we give an averaging formula in the same flavor as in \cite{Ma}, which enables us to use the concept of "goodness" of cubes in our estimate. Then, we will move on to the main body of this paper, prove the $T(b)$ theorem by a case by case estimate of terms in the averaging formula.

\section*{Acknowledgement} 
The author would like to thank Jill Pipher for guiding her into this area, suggesting the topic and the numerous fruitful discussions. The author is also grateful to Michael Lacey and Brett Wick for useful discussions during her visit to Georgia Institute of Technology.

\section{Assumptions on the operator}\label{assump}
\subsection*{Bi-parameter $b$-adapted martingale differences}
As a preliminary, we begin with a quick introduction of the martingale difference decomposition adapted to our problem.

Let $\omega^n= (\omega^n_i)_{i \in \Z}$, where $\omega^n_i \in \{0,1\}^n$. Let $\mathcal{D}^n_0$ be the standard dyadic grid on $\R^n$.
We define the shifted dyadic grid $\mathcal{D}^n_{\omega^n} = \{I + \sum_{i:\, 2^{-i} < \ell(I)} 2^{-i}\omega^n_i: \, I \in \mathcal{D}^n_0\} = \{I \dotplus \omega^n: \, I \in \mathcal{D}^n_0\}$, where
$I \dotplus \omega^n := I + \sum_{i:\, 2^{-i} < \ell(I)} 2^{-i}\omega^n_i$. There is a natural probability structure on $(\{0,1\}^n)^{\Z}$, which gives us a random dyadic grid $\mathcal{D}^n_{\omega^n}$ in $\R^n$. When there is no need to specify what is the $\omega^n$, most of the time, we just write $\mathcal{D}^n$ for short. Interested readers can find more detailed discussion of random dyadic grids in \cite{Hy} or \cite{Ma}.

Given a pseudo-accretive function $b=b_1\otimes b_2$, and two fixed dyadic grids $\mathcal{D}^n, \mathcal{D}^m$ in $\R^n, \R^m$, respectively. For each $p\in\Z$, let $\mathcal{D}^n_p$ be the collection of cubes of side length $2^{-p}$ in $\mathcal{D}^n$, we have 
\begin{displaymath}
E^{b_1}_pf=\sum_{I\in\mathcal{D}^n_p}\frac{\int_I fb_1}{\int_I b_1}\chi_I,\quad E^{b_1}_If=\chi_I E^{b_1}_pf.
\end{displaymath}
Similarly, we have $E^{b_2}_q$ and $E^{b_2}_J$ defined for each $q\in\Z, J\in\mathcal{D}^m$. Then their composition is a $b$-adapted double expectation operator:
\begin{displaymath} 
E^b_{p,q}=E^{b_1}_pE^{b_2}_q=E^{b_2}_qE^{b_1}_p.
\end{displaymath}

Let $\Delta^{b_1}_p=E^{b_1}_{p+1}-E^{b_1}_p$, $\Delta^{b_1}_I=\chi_I\Delta^{b_1}_p$ for each $I\in\mathcal{D}^n_p$, and similarly for the other variable.  The $b$-adapted double martingale difference is defined as
\begin{displaymath}
\Delta^b_{p,q}=\Delta^{b_1}_p\Delta^{b_2}_q=\Delta^{b_2}_q\Delta^{b_1}_p.
\end{displaymath}

The following properties can be easily checked:

\begin{enumerate}
\item $\Delta^b_{I\times J}f$ is supported on the dyadic rectangle $I\times J$, and is a constant on each of its children;\\
\item $\int b_1\Delta^b_{p,q}f\,dx_1=\int b_2\Delta^b_{p,q}f\,dx_2=0$;\\
\item $\Delta^b_{p,q}\Delta^b_{k,l}=0$ unless $p=k, q=l$, and in this case it equals $\Delta^b_{p,q}$;\\
\item If $f\in L^2(\R^n\times\R^m)$, then $f=\sum_{p,q}\Delta^b_{p,q}f$ with convergence in $L^2$, and
\begin{displaymath}
\|f\|^2_{L^2}\lesssim\sum_{p,q}\|\Delta^b_{p,q}f\|^2_{L^2}\lesssim\|f\|^2_{L^2}.
\end{displaymath}
\end{enumerate}

Property $(4)$ can be verified by iteration of the one-parameter martingale difference argument in \cite{NTV}.

Moreover, we observe that
\begin{displaymath}
E^{b*}_{p,q}f=E^{b_1*}_pE^{b_2*}_qf=b\sum_{I\in\mathcal{D}^n_p,J\in\mathcal{D}^m_q}\frac{\int_{I\times J}f}{\int_{I\times J}b}\chi_{I\times J}.
\end{displaymath}
and hence
\begin{displaymath}
M_b\Delta^b_{p,q}=\Delta^{b*}_{p,q}M_b,
\end{displaymath}
where $M_bf=bf$ is the multiplication operator by $b$.

We now introduce the assumptions on $T$ that we will need throughout the argument. Fix two pseudo-accretive functions $b=b_1\otimes b_2, b'=b'_1\otimes b'_2$. For simplicity, denote $d=b_1\otimes b'_2$ and $d'=b'_1\otimes b_2$, then obviously $d, d'$ are also pseudo-accretive. 

\subsection*{Full Calder\'on-Zygmund structure}

If $f=f_1\otimes f_2$ and $g=g_1\otimes g_2$ with $f_1, g_1\in C^\infty_0({\R^n})$, $f_2, g_2\in C^\infty_0({\R^m})$, $\mbox{spt} f_1\cap \mbox{spt} g_1=\emptyset$ and $\mbox{spt} f_2\cap\mbox{spt} g_2=\emptyset$, then we have the full kernel representation
\begin{displaymath}
\pair{M_{b'}TM_bf}{g}=\int_{\R^{n+m}}\int_{\R^{n+m}} K(x,y)f(y)g(x)b(y)b'(x)\,dxdy.
\end{displaymath}

The kernel $K: (\R^{n+m}\times\R^{n+m})\setminus \{(x,y)\in\R^{n+m}\times\R^{n+m}:\,x_1=y_1 \;\mbox{or}\; x_2=y_2\}\rightarrow \C$ is assumed to satisfy 

\begin{enumerate}
\item Size condition
\begin{displaymath}
|K(x,y)|\leq C\frac{1}{|x_1-y_1|^n}\frac{1}{|x_2-y_2|^m}.
\end{displaymath}

\item H\"{o}lder conditions

\begin{displaymath}
|K(x,y)-K(x,(y_1,y_2'))-K(x,(y_1',y_2))+K(x,y')|\leq C\frac{|y_1-y_1'|^\delta}{|x_1-y_1|^{n+\delta}}\frac{|y_2-y_2'|^\delta}{|x_2-y_2|^{m+\delta}}
\end{displaymath}
whenever $|y_1-y_1'|\leq|x_1-y_1|/2$ and $|y_2-y_2'|\leq|x_2-y_2|/2$,

\begin{displaymath}
|K(x,y)-K((x_1,x_2'),y)-K((x_1',x_2),y)+K(x',y)|\leq C\frac{|x_1-x_1'|^\delta}{|x_1-y_1|^{n+\delta}}\frac{|x_2-x_2'|^\delta}{|x_2-y_2|^{m+\delta}}
\end{displaymath}
whenever $|x_1-x_1'|\leq|x_1-y_1|/2$ and $|x_2-x_2'|\leq|x_2-y_2|/2$,

\begin{displaymath}
|K(x,y)-K((x_1,x_2'),y)-K(x,(y_1',y_2))+K((x_1,x_2'),(y_1',y_2))|\leq C\frac{|y_1-y_1'|^\delta}{|x_1-y_1|^{n+\delta}}\frac{|x_2-x_2'|^\delta}{|x_2-y_2|^{m+\delta}}
\end{displaymath}
whenever $|y_1-y_1'|\leq|x_1-y_1|/2$ and $|x_2-x_2'|\leq|x_2-y_2|/2$,

\begin{displaymath}
|K(x,y)-K(x,(y_1,y_2'))-K((x_1',x_2),y)+K((x_1',x_2),(y_1,y_2'))|\leq C\frac{|x_1-x_1'|^\delta}{|x_1-y_1|^{n+\delta}}\frac{|y_2-y_2'|^\delta}{|x_2-y_2|^{m+\delta}}
\end{displaymath}
whenever $|x_1-x_1'|\leq|x_1-y_1|/2$ and $|y_2-y_2'|\leq|x_2-y_2|/2$.

\item Mixed H\"{o}lder-size conditions

\begin{displaymath}
|K(x,y)-K((x_1',x_2),y)|\leq C\frac{|x_1-x_1'|^\delta}{|x_1-y_1|^{n+\delta}}\frac{1}{|x_2-y_2|^m}
\end{displaymath}
whenever $|x_1-x_1'|\leq|x_1-y_1|/2$,

\begin{displaymath}
|K(x,y)-K(x,(y_1',y_2))|\leq C\frac{|y_1-y_1'|^\delta}{|x_1-y_1|^{n+\delta}}\frac{1}{|x_2-y_2|^m}
\end{displaymath}
whenever $|y_1-y_1'|\leq|x_1-y_1|/2$,

\begin{displaymath}
|K(x,y)-K((x_1,x_2'),y)|\leq C\frac{1}{|x_1-y_1|^n}\frac{|x_2-x_2'|^\delta}{|x_2-y_2|^{m+\delta}}
\end{displaymath}
whenever $|x_2-x_2'|\leq|x_2-y_2|/2$,

\begin{displaymath}
|K(x,y)-K(x,(y_1,y_2'))|\leq C\frac{1}{|x_1-y_1|^n}\frac{|y_2-y_2'|^\delta}{|x_2-y_2|^{m+\delta}}
\end{displaymath}
whenever $|y_2-y_2'|\leq|x_2-y_2|/2$.
\end{enumerate}

\subsection*{Partial Calder\'on-Zygmund structure}

We also need some C-Z structure on $\R^n$ and $\R^m$ separately to deal with the case when $f,g$ are only separated on one variable. If $f=f_1\otimes f_2, g=g_1\otimes g_2$ and $\mbox{spt} f_1\cap\mbox{spt} g_1=\emptyset$, then we have the partial kernel representation
\begin{displaymath}
\pair{M_{b'}TM_bf}{g}=\int_{\R^n}\int_{\R^n}K_{f_2,g_2}(x_1,y_1)f_1(y_1)g_1(x_1)b_1(y_1)b'_1(x_1)\,dx_1dy_1.
\end{displaymath}

The partial kernel $K_{f_2,g_2}$ defined on $(\R^n\times\R^n)\setminus\{(x_1,y_1)\in\R^n\times\R^n:\,x_1=y_1\}$ is assumed to satisfy the following standard estimates:

\begin{enumerate}
\item Size condition
\begin{equation}\label{parsize}
|K_{f_2,g_2}|\leq C(f_2,g_2)\frac{1}{|x_1-y_1|^n}.
\end{equation}
\item H\"{o}lder conditions
\begin{equation}\label{parH1}
|K_{f_2,g_2}(x_1,y_1)-K_{f_2,g_2}(x_1',y_1)|\leq C(f_2,g_2)\frac{|x_1-x_1'|^\delta}{|x_1-y_1|^{n+\delta}}
\end{equation}
whenever $|x_1-x_1'|\leq |x_1-y_1|/2$,
\begin{equation}\label{parH2}
|K_{f_2,g_2}(x_1,y_1)-K_{f_2,g_2}(x_1,y_1')|\leq C(f_2,g_2)\frac{|y_1-y_1'|^\delta}{|x_1-y_1|^{n+\delta}}
\end{equation}
whenever $|y_1-y_1'|\leq |x_1-y_1|/2$.
\end{enumerate}

This assumption is in the same flavor of \cite{Ma}, and is important of defining $T(b)$. In fact, we can weaken this by assuming the above only for the cases when
\begin{displaymath}
(f_2,g_2)=(\chi_V,\chi_V),\quad (\chi_V,u_Vb^{'-1}_2),\quad\mbox{or}\quad (u_Vb^{-1}_2,\chi_V), 
\end{displaymath} 
for any cube $V\subset\R^m$, and $u_V$ being a $V$-adapted function with zero-mean (i.e. $\mbox{spt}u_V\subset V$, $|u_V|\leq 1$ and $\int u_V=0$).

We also need to assume that there exists a universal constant $C$, such that  
\begin{displaymath}
C(\chi_V,\chi_V)+C(\chi_V,u_Vb'^{-1}_2)+C(u_Vb_2^{-1},\chi_V)\leq C|V|.
\end{displaymath}

It is easily shown that both full and partial kernel representations also hold when $f,g$ are finite linear combinations of characteristic functions, or even tensor products of compactly supported $L^\infty$ functions, as long as for the required variable, they are still disjointly supported. To see this, when taking those functions, following from the standard condition on the kernels, both integrals are still convergent. We can use them to define the corresponding bilinear forms. After we finally show that $T$ is bounded on $L^2$ (here we don't even need the boundedness assumption on $T$ a priori), use the density of $C^\infty_0$ functions and Lebesgue dominated convergence theorem, we can show that the bilinear form has to be equal to the kernel representation, hence is well defined.

The  partial C-Z structure assumption is natural. Recall how Journ\'e defined his class of operators in \cite{Jo}. Rephrasing in terms of our definition, Journ\'e assumed that the partial kernel $K_{f_2,g_2}(x_1,y_1)$ is a bilinear form associated with a $\mathcal{L}(L^2(\R^m),L^2(\R^m))$ valued standard C-Z kernel, which then implies the size and H\"older conditions (\ref{parsize}), (\ref{parH1}), (\ref{parH2}). In the bi-parameter setting, the partial C-Z structure assumptions are required to both define $Tb$ and to handle the "mixed cases". That arise because of the independent behavior in each variable. (See Section 6, 7, 9, 12 for discussions of different "mixed cases"). As far as we know, all the previous literature in this area needs some assumptions about the partial C-Z structure of the operator. For example, in Pott and Villarroya's most recent version of \cite{PV}, they included such an assumption on the operator so that they can fully justify the definition of $T1$. Although it is formulated a little differently, but is in spirit the same as ours. Martikainen (\cite{Ma}) also requires a similar assumption. (See Section 2 of \cite{Ma}).

Note that in the case $f,g$ are separated in both variables, i.e. when we have the full kernel representation, the partial kernels are just
\begin{displaymath}
K_{f_2,g_2}(x_1,y_1)=\int_{\R^m}\int_{\R^m}K(x,y)f_2(y_2)g_2(x_2)b_2(y_2)b'_2(x_2)\,dx_2dy_2,
\end{displaymath}
and both of the size and H\"older conditions follow easily.

We also assume that the symmetric partial kernel representation and corresponding conditions on kernel $K_{f_1,g_1}$ in the case $\mbox{spt}f_2\cap\mbox{spt}g_2=\emptyset$.

\subsection*{Weak boundedness property}
We assume that there exists a constant $C$ such that, for any cube $K\subset\R^n$ and $V\subset\R^m$,
\begin{displaymath}
|\pair{M_{b'}TM_b(\chi_K\otimes\chi_V)}{\chi_K\otimes\chi_V}|\leq C|K||V|.
\end{displaymath}

\subsection*{BMO conditions}
We assume $Tb, T^*b', T_1d', T_1^*d\in BMO(\R^n\times\R^m)$, where $T_1$ is the partial adjoint of $T$ defined by
\begin{displaymath}
\pair{T_1(f_1\otimes f_2)}{g_1\otimes g_2}=\pair{T(g_1\otimes f_2)}{f_1\otimes g_2}.
\end{displaymath}

Here, by assuming that they are in $BMO(\R^n\times\R^m)$, equivalently, we mean that they are in $BMO_d(\R^n\times\R^m)$, the dyadic $BMO$ space for any dyadic grid. It is proved by Pipher and Ward \cite{PW} that in the bi-parameter setting, the product $BMO$ is the average of dyadic $BMO$. This result is then reproved and extended to multi-parameter by Treil \cite{Tr} through a different method. We now run into a problem of defining $Tb$ (and similarly for the other three functions). In order to do this, we are going to show that $Tb$ lies in the dual of some properly selected subspace $A$ of $H^1_d(\R^n\times\R^m)$, i.e. the bilinear form $\pair{g}{Tb}$ is well defined for any $g\in A$.

Let $A$ be the space consisting of all the functions
\begin{displaymath}
b'\sum_{\mbox{finite}I,J}\Delta_I^{b'_1}\Delta_J^{b'_2}f
\end{displaymath}
where $f\in C^\infty_0(\R^n\times\R^m)$, $I\in\mathcal{D}^n, J\in\mathcal{D}^m$ and we are summing over a finite number of terms. It is easily seen that $A$ is indeed a subspace of $H^1_d(\R^n\times\R^m)$. Hence, by linearity, it suffices to define $\pair{b'\Delta_I^{b'_1}\Delta_J^{b'_2}f}{Tb}$.

Divide the bilinear form into four parts:
\begin{displaymath}
\begin{split}
&\pair{b'\Delta_I^{b'_1}\Delta_J^{b'_2}f}{T(b\chi_{3I}\otimes\chi_{3J})}+\pair{b'\Delta_I^{b'_1}\Delta_J^{b'_2}f}{T(b\chi_{3I}\otimes\chi_{(3J)^c})}\\
&+\pair{b'\Delta_I^{b'_1}\Delta_J^{b'_2}f}{T(b\chi_{(3I)^c}\otimes\chi_{3J})}+\pair{b'\Delta_I^{b'_1}\Delta_J^{b'_2}f}{T(b\chi_{(3I)^c}\otimes\chi_{(3J)^c})}.
\end{split}
\end{displaymath}

Part one: $\Delta_I^{bÔ_1}\Delta_J^{bÕ_2}f$ is a finite linear combination of characteristic functions. For each $I_i\in\mbox{ch}(I),J_j\in\mbox{ch}(J)$, 
\begin{displaymath}
\pair{b'\chi_{I_i}\otimes\chi_{J_j}\Delta_I^{b'_1}\Delta_J^{b'_2}f}{T(b\chi_{3I}\otimes\chi_{3J})}=\Delta_I^{b'_1}\Delta_J^{b'_2}f|_{I_i\times J_j}\pair{bÔ\chi_{I_i}\otimes\chi_{J_j}}{T(b\chi_{3I}\otimes\chi_{3J})},
\end{displaymath}

and
\begin{displaymath}
\begin{split}
&\pair{b'\chi_{I_i}\otimes\chi_{J_j}}{T(b\chi_{3I}\otimes\chi_{3J})}\\
&=\pair{b'\chi_{I_i}\otimes\chi_{J_j}}{T(b\chi_{I_i}\otimes\chi_{J_j})}+\pair{b'\chi_{I_i}\otimes\chi_{J_j}}{T(b\chi_{I_i}\otimes\chi_{3J\setminus J_j})}\\
&\qquad+\pair{b'\chi_{I_i}\otimes\chi_{J_j}}{T(b\chi_{3I\setminus I_i}\otimes\chi_{J_j})}+\pair{b'\chi_{I_i}\otimes\chi_{J_j}}{T(b\chi_{3I\setminus I_i}\otimes\chi_{3J\setminus J_j})}.
\end{split}
\end{displaymath}
In the above, the first term makes sense due to the weak boundedness property. The second and third terms can be dealt with using partial kernel representation. Finally, the last term can be defined using full kernel representation.

Part two (and similarly for part three): Write 
\begin{displaymath}
\begin{split}
\pair{b'\Delta_I^{b'_1}\Delta_J^{b'_2}f}{T(b\chi_{3I}\otimes\chi_{(3J)^c})}&=\sum_{i=1}^{2^n}\pair{b'\chi_{I_i}\Delta_I^{b'_1}\Delta_J^{b'_2}f}{T(b\chi_{3I}\otimes\chi_{(3J)^c})}\\
&=\sum_{i=1}^{2^n}\pair{b'\chi_{I_i}(x_1)\otimes\Delta_I^{b'_1}\Delta_J^{b'_2}f|_{I_i}(x_2)}{T(b\chi_{3I}\otimes\chi_{(3J)^c})}.
\end{split}
\end{displaymath}
Then for each term in the above, since the functions have good separation on one variable, we know that in the case that everything is compactly supported, it has a partial kernel representation:
\begin{displaymath}
\begin{split}
&\pair{b'\chi_{I_i}(x_1)\otimes\Delta_I^{b'_1}\Delta_J^{b'_2}f(x_2)}{T(b\chi_{3I}\otimes\chi_{(3J)^c})}\\
&=\int_{(3J)^c}\int_J K_{\chi_{3I},\chi_{I_i}}(x_2,y_2)b_2(y_2)\Delta_I^{b'_1}\Delta_J^{b'_2}f|_{I_i}(x_2)b'_2(x_2)\,dx_2dy_2\\
&=\int_{(3J)^c}\int_J\left[K_{\chi_{3I},\chi_{I_i}}(x_2,y_2)-K_{\chi_{3I},\chi_{I_i}}(c_J,y_2)\right]b_2(y_2)\Delta_I^{b'_1}\Delta_J^{b'_2}f|_{I_i}(x_2)b'_2(x_2)\,dx_2dy_2,
\end{split}
\end{displaymath}

While the integrand is not compactly supported, and the H\"{o}lder condition for partial kernels implies that the integral is convergent, it can be used to serve as the definition of the bilinear form on the left hand side.

Part four: In this part, the functions have good separations on both variables. As above, although we don't have a full kernel representation for the bilinear form directly due to the fact that the integrand is not compactly supported, we can define it as follows:
\begin{displaymath}
\int_{(3I)^c\times(3J)^c}\int_{I\times J} K(x,y)b(y)b'(x)\Delta_I^{b'_1}\Delta_J^{b'_2}f(x)\,dxdy
\end{displaymath}
and prove that the integral does converge. To see this last fact, we change $K(x,y)$ to
\begin{displaymath}
K(x,y)-K(c_I,x_2,y)-K(x_1,c_J,y)+K(c_I,c_J,y)
\end{displaymath}
by cancellation. Then the H\"{o}lder condition for the full kernel will imply the convergence of the integral.

Note that in parts two, three and four, we don't give an arbitrary definition to those bilinear forms. A simple limiting argument shows that they are well defined. Consider part four for example. Let $\varphi$ be a cut-off function, such that $\varphi=1$ on $I\times J$, and $\varphi=0$ outside $3I\times 3J$. Denote dilation $D_{k_1,k_2}\varphi(x)=\varphi(x_1k_1^{-1},x_2k_2^{-1})$. Since $\Delta_I^{b'_1}\Delta_J^{b'_2}f$ is a finite linear combination of characteristic functions, by the linearity of bilinear forms and full kernel representations, we have
\begin{displaymath}
\begin{split}
&\pair{b'\Delta_I^{b'_1}\Delta_J^{b'_2}f}{T(bD_{k_1,k_2}\varphi\chi_{(3I)^c}\otimes\chi_{(3J)^c})}\\
&=\int_{(3I)^c\times(3J)^c}\int_{I\times J} K(x,y)D_{k_1,k_2}\varphi(y)b(y)b'(x)\Delta_I^{b'_1}\Delta_J^{b'_2}f(x)\,dxdy.
\end{split}
\end{displaymath}
Changing the kernel and using the H\"{o}lder condition for the full kernel as above, together with the boundedness of $f$ and $\varphi$, we can show that the integrand is uniformly bounded by a constant multiple of
\begin{displaymath}
\frac{1}{|x_1-y_1|^{n+\delta}}\frac{1}{|x_2-y_2|^{m+\delta}}.
\end{displaymath}
Then the Lebesgue dominated convergence theorem implies that
\begin{displaymath}
\begin{split}
&\pair{b'\Delta_I^{b'_1}\Delta_J^{b'_2}f}{T(b\chi_{(3I)^c}\otimes\chi_{(3J)^c})}\\
&=\lim_{k_1,k_2\rightarrow\infty}\pair{b'\Delta_I^{b'_1}\Delta_J^{b'_2}f}{T(bD_{k_1,k_2}\varphi\chi_{(3I)^c}\otimes\chi_{(3J)^c})}\\
&=\int_{(3I)^c\times(3J)^c}\int_{I\times J} K(x,y)b(y)b'(x)\Delta_I^{b'_1}\Delta_J^{b'_2}f(x)\,dxdy.
\end{split}
\end{displaymath}
And it's easily seen that the above definition is independent of the choice of $\varphi$.

Hence, $Tb$ lies in the dual of $A$. By saying that it belongs to $BMO_d(\R^n\times\R^m)$, we mean that it is bounded on $A$ and can be boundedly extended to a functional defined on the whole $H^1_d(\R^n\times\R^m)$. And we can use the same technique above to give meanings to the other three objects similarly. Note that we can actually weaken this $BMO$ assumption by only assuming that $T(b)$ is a functional on $A$, and similarly for the other three (but with differently chosen subspaces of $H^1(\R^n\times \R^m)$). We will see in the following that this is all we need.

\subsection*{Diagonal BMO conditions}
There exists constant $C$ such that, for any cube $K\subset\R^n$, $V\subset\R^m$, and any zero-mean functions $a_K$, $b_V$ which are $K$, $V$ adapted, respectively, the following hold:
\begin{itemize}
\item $|\pair{M_{b'}TM_b(a_Kb_1^{-1}\otimes\chi_V)}{\chi_K\otimes\chi_V}|\leq C|K||V|$
\item $|\pair{M_{b'}TM_b(\chi_K\otimes\chi_V)}{a_Kb'^{-1}_1\otimes\chi_V}|\leq C|K||V|$
\item $|\pair{M_{b'}TM_b(\chi_K\otimes b_Vb_2^{-1})}{\chi_K\otimes\chi_V}|\leq C|K||V|$
\item $|\pair{M_{b'}TM_b(\chi_K\otimes\chi_V)}{\chi_K\otimes b_Vb'^{-1}_2}|\leq C|K||V|$
\end{itemize}

\section{Bi-parameter $b$-adapted paraproducts}

In this section, we will discuss the boundedness of three different kinds of bi-parameter $b$-adapted paraproducts that will be used in the proof of our $T(b)$ theorem.

\subsection*{Partial paraproducts}
By partial paraproduct we mean a classical one-parameter $b$-adapted paraproduct with respect to one variable. 

\begin{defn}
Let $a\in BMO(\R^m)$. Then, for two fixed pseudo-accretive functions $b_2,b'_2\in L^\infty(\R^m)$, the operator $\pi^{b'_2,b_2}_a$ is a partial paraproduct, acting on functions on $\R^m$:
\begin{displaymath}
\pi^{b'_2,b_2}_a(f)=\sum_{V\in\mathcal{D}^m}\ave{f}^{b'_2}_VM_{b_2}\Delta^{b_2}_Va.
\end{displaymath}
Similarly, there is a symmetric partial paraproduct with respect to the other variable for fixed pseudo-accretive functions $b_1,b'_1\in L^\infty(\R^n)$, acting on functions on $\R^n$.
\end{defn}

\begin{prop}
Partial paraproducts are bounded operators on $L^2$. Specifically,
\begin{displaymath}
\|\pi^{b'_2,b_2}_a(f)\|_{L^2(\R^m)}\lesssim\|a\|_{BMO}\|f\|_{L^2(\R^m)},
\end{displaymath}
and a similar inequality holds for the symmetric one.
\end{prop}

\begin{proof}
We only prove the first inequality. For any $f,g\in L^2(\R^m)$,
\begin{displaymath}
\begin{split}
|\pair{\pi^{b'_2,b_2}_a(f)}{g}|&=|\pair{\sum_{V\in\mathcal{D}^m}\ave{f}^{b'_2}_VM_{b_2}\Delta^{b_2}_Va}{g}|\\
&=|\pair{\Delta^{b_2*}_V\sum_{V}\ave{f}^{b'_2}_VM_{b_2}\Delta^{b_2}_Va}{g}|\\
&=|\pair{\sum_{V}\ave{f}^{b'_2}_VM_{b_2}\Delta^{b_2}_Va}{\Delta^{b_2}_Vg}|\\
&\lesssim\sum_{V}\|\ave{f}^{b'_2}_V\Delta^{b_2}_Va\|_{L^2(\R^m)}\|\Delta^{b_2}_Vg\|_{L^2(\R^m)}\\
&\leq (\sum_{V}\|\ave{f}^{b'_2}_V\Delta^{b_2}_Va\|^2_{L^2(\R^m)})^{1/2}(\sum_{V}\|\Delta^{b_2}_Vg\|^2_{L^2(\R^m)})^{1/2}\\
&\lesssim (\sum_{V}|\ave{f}^{b'_2}_V|^2\|\Delta^{b_2}_Va\|^2_{L^2(\R^m)})^{1/2}\|g\|_{L^2(\R^m)}.
\end{split}
\end{displaymath}
where the fourth and fifth lines follow from H\"{o}lder inequality. Hence, it suffices to show that
\begin{displaymath}
\sum_{V}|\ave{f}^{b'_2}_V|^2\|\Delta^{b_2}_Va\|^2_{L^2(\R^m)}\lesssim \|a\|^2_{BMO}\|f\|^2_{L^2(\R^m)}.
\end{displaymath}
To see this, by the boundedness of $b'_2$,
\begin{displaymath}
|\ave{f}^{b'_2}_V|\lesssim |V|^{-1}\int_V|f|=\ave{|f|}_V.
\end{displaymath}
Hence, it suffices to prove
\begin{displaymath}
\sum_V|\ave{|f|}_V|^2\|\Delta^{b_2}_Va\|^2_{L^2}\lesssim \|a\|^2_{BMO}\||f|\|^2_{L^2}.
\end{displaymath}
Observing the above inequality, we see that by Carleson embedding theorem, all we need is to show that $\{\|\Delta^{b_2}_Va\|^2_{L^2}\}_V$ is a Carleson sequence with constant $\lesssim \|a\|^2_{BMO}$, i.e.
\begin{displaymath}
\forall J\in\mathcal{D}^m,\quad \sum_{I\in\mathcal{D}(J)}\|\Delta^{b_2}_Ia\|^2_{L^2}\lesssim\|a\|^2_{BMO} |J|.
\end{displaymath}
And this is not hard to prove since the $b$-adapted martingale differences satisfy the $L^2$ property by \cite{NTV}. Indeed, since $\|a\|^2_{BMO}=\sup_J\frac{1}{|J|}\int_J|a-\ave{a}_J|^2<\infty$, for any fixed dyadic $J$,
\begin{displaymath}
\begin{split}
\|a\|^2_{BMO}|J|&\geq \int_J|a-\ave{a}_J|^2=\|\chi_J(a-\ave{a}_J)\|^2_{L^2}\\
&\approx \sum_{I}\|\Delta^{b_2}_I(\chi_J(a-\ave{a}_J))\|^2_{L^2}\\
&\geq \sum_{I\in\mathcal{D}(J)}\|\Delta^{b_2}_I(\chi_J(a-\ave{a}_J))\|^2_{L^2}=\sum_{I\in\mathcal{D}(J)}\|\Delta^{b_2}_I(a-\ave{a}_J)\|^2_{L^2}\\
&=\sum_{I\in\mathcal{D}(J)}\|\Delta^{b_2}_Ia\|^2_{L^2}.
\end{split}
\end{displaymath}
where the last equality is because $\Delta^{b_2}_I$ maps any constant function to 0. And this completes the proof.
\end{proof}

\subsection*{Full paraproducts}
We now introduce a "real" bi-parameter $b$-adapted paraproduct, which is a natural generalization of the classical one-parameter one.

\begin{defn}
For $a\in BMO(\R^n\times\R^m)$, operator $\pi^{b',b}_a$ is called full paraproduct, defined as
\begin{displaymath}
\pi^{b',b}_a(f)=\sum_{K\in\mathcal{D}^n,V\in\mathcal{D}^m}\ave{f}^{b'}_{K\times V}M_b\Delta^{b_1}_K\Delta^{b_2}_Va.
\end{displaymath}
\end{defn}

\begin{prop}\label{fullpara}
Full paraproducts are bounded operators on $L^2(\R^n\times\R^m)$. Specifically,
\begin{displaymath}
\|\pi^{b',b}_a(f)\|_{L^2(\R^n\times\R^m)}\lesssim \|a\|_{BMO(\R^n\times\R^m)}\|f\|_{L^2(\R^n\times\R^m)}.
\end{displaymath}
\end{prop}

To prove this proposition, we need to first consider the space $H^1_b(\R^n\times\R^m)$, containing those functions $f$ such that $fb\in H^1(\R^n\times\R^m)$. It is easy to check that the dual space of $H^1_b(\R^n\times\R^m)$ is $BMO_b(\R^n\times\R^m)$, containing functions $f$ such that $fb^{-1}\in BMO(\R^n\times\R^m)$. It is well known that $H^1$ can be characterized using both martingale maximal function and square function with the norms being equivalent (\cite{CFS}). Similarly, if we define a $b$-adapted maximal function
\begin{displaymath}
f^*_b(x)=\sup_{p,q\in\Z}|E^{b_1}_p E^{b_2}_q f(x)|=\sup_{I\in\mathcal{D}^n,J\in\mathcal{D}^m}|E^{b_1}_I E^{b_2}_J f(x)|,
\end{displaymath}
then, we have the following fact
\begin{prop}
A function $f$ belongs to $H^1_b(\R^n\times\R^m)$ if and only if $f^*_b\in L^1(\R^n\times\R^m)$.
\end{prop}

Now, define a $b$-adapted square function as
\begin{displaymath}
S_bf(x)=(\sum_{p,q\in\Z} |\Delta^{b_1}_p\Delta^{b_2}_q f(x)|^2)^{1/2}=(\sum_{I\in\mathcal{D}^n,J\in\mathcal{D}^m} |\Delta^{b_1}_I\Delta^{b_2}_J f(x)|^2)^{1/2},
\end{displaymath}
and let the space $K^1_b(\R^n\times\R^m)$ consist of all the functions $f$ such that $S_bf\in L^1(\R^n\times\R^m)$. We then have the following theorem.

\begin{thm}\label{squaremax}
If $f\in K^1_b(\R^n\times\R^m)$, then $f\in H^1_b(\R^n\times\R^m)$. Moreover, for all $f\in K^1_b$, $\|f^*_b\|_{L^1}\lesssim \|S_bf\|_{L^1}$.
\end{thm}

To prove Theorem \ref{squaremax}, we use the idea of double martingale by Bernard and a technique involving atomic decomposition. See \cite{Be}. 

First, in our $b$-adapted case, the well known equivalence of $L^2$ norm between martingale maximal function and square function is still true. More specifically, we have
\begin{prop}
If function $f\in L^2(\R^n\times\R^m)$, then both $f^*_b$ and $S_bf$ are in $L^2$, and their norms are equivalent to $\|f\|_{L^2}$.
\end{prop}

\begin{proof}
Iteration of a well known one-parameter $L^2$ result (see \cite{NTV}) gives 
\begin{displaymath}
\|f\|^2_{L^2}\approx \sum_{p,q}\|\Delta^{b_1}_p\Delta^{b_2}_qf\|^2_{L^2}.
\end{displaymath}
Hence,
\begin{displaymath}
\|S_bf\|^2_{L^2}=\int |S_bf|^2=\sum_{p,q}\int |\Delta^{b_1}_p\Delta^{b_2}_qf|^2=\sum_{p,q}\|\Delta^{b_1}_p\Delta^{b_2}_qf\|^2_{L^2}\approx\|f\|^2_{L^2}.
\end{displaymath}
For martingale maximal function, $f\leq f^*_b\,a.e.$ gives $\|f\|_{L^2}\leq\|f^*_b\|_{L^2}$. On the other hand, by accretivity
\begin{displaymath}
f^*_b=\sup_{I,J}|E^{b_1}_IE^{b_2}_Jf|\lesssim \sup_{I,J}\frac{\int_{I\times J}|f|}{|I\times J|}\leq M^Sf,
\end{displaymath}
and the strong maximal function is bounded on $L^2$, it implies $\|f^*_b\|_{L^2}\lesssim\|f\|_{L^2}$.
\end{proof}

For simplicity, denote $f_{p,q}=E^{b_1}_pE^{b_2}_q f$, and for each pair $(p,q)\in\Z\times\Z$, let $\mathcal{F}_{p,q}$ be the $\sigma$-algebra generated by all the dyadic rectangles of size $2^{-p}\times 2^{-q}$.
\begin{defn}
The function $x\mapsto\tau(x)\subset\Z\times\Z$ is called a stopping time if $\{x:\,(p,q)\in\tau(x)\}$ is $\mathcal{F}_{p,q}$-measurable.
\end{defn}

\begin{defn}
$a\in L^2(\R^n\times\R^m)$ is \emph{an atom} of $K^1_b$ if there exists a stopping time $\tau$ such that
\begin{enumerate}
\item $|\{x:\,\tau(x)\neq \Z\times\Z\}|<\infty$;
\item Let $a_t=E^b_ta,\,t\in\Z\times\Z$, then $a_t(x)=0,\,\forall t+1\in\tau(x)$;
\item $\|a^*_b\|_{L^2}\lesssim |\{x:\,\tau(x)\neq \Z\times\Z\}|^{-1/2}$.
\end{enumerate}
\end{defn}

Note that if we call $F=\{x:\,\tau(x)\neq \Z\times\Z\}$, then from property $(2)$ in the definition, both $a^*_b$ and $S_ba$ are supported on $F$. Also, such functions are called atoms because they have the following property.

\begin{prop}
If $a$ is an atom, then $a\in C\cdot\mathrm{B}$, where $\mathrm{B}$ is the unit ball in $H^1_b$ or $K^1_b$, and $C$ is a universal constant independent of $a$.
\end{prop}

\begin{proof}
Using the supports of $a^*_b$ and $S_ba$, H\"{o}lder inequality implies
\begin{displaymath}
\|a^*_b\|_{L^1}=\|a^*_b\chi_{F}\|_{L^1}\leq\|a^*_b\|_{L^2} |F|^{1/2}\lesssim 1.
\end{displaymath}
and
\begin{displaymath}
\|S_ba\|_{L^1}=\|S_ba\chi_F\|_{L^1}\leq\|S_ba\|_{L^2} |F|^{1/2}\approx \|a^*_b\|_{L^2} |F|^{1/2}\lesssim 1.
\end{displaymath}
\end{proof}

We now state the theorem of atomic decomposition.
\begin{thm}\label{atom}
Given $f\in K^1_b\cap L^2$, there exists a sequence of atoms $a^n$ and a sequence of scalars $\lambda_n$ such that
\begin{enumerate}
\item $f=\sum_n \lambda_n a^n,\,a.e.$
\item $\sum_n|\lambda_n|\lesssim \|f\|_{K^1_b}$.
\end{enumerate}
\end{thm}

Before stating the proof of Theorem \ref{atom}, we show that this atomic decomposition result will imply Theorem \ref{squaremax}.

\begin{proof}(of Theorem \ref{squaremax})
It suffices to show the result holds for $f\in L^2$. For any such function, atomic decomposition implies 
\begin{displaymath}
f_t=\sum_n\lambda_n a^n_t,\,a.e.\,\forall t\in\Z\times\Z.
\end{displaymath}
Then,
\begin{displaymath}
f^*_b=\sup_t|f_t|=\sup_t|\sum_n\lambda_n a^n_t|\leq\sum_n |\lambda_n|\sup_t|a^n_t|=\sum_n|\lambda_n|(a^n)^*_b,
\end{displaymath}
which implies
\begin{displaymath}
\|f^*_b\|_{L^1}\leq\sum_n|\lambda_n|\sup_n\|(a^n)^*_b\|_{L^1}\lesssim\sum_n|\lambda_n|\lesssim\|S_bf\|_{L^1}.
\end{displaymath}
\end{proof}

We turn to the prove of Theorem \ref{atom}.

\begin{proof}(of Theorem \ref{atom})
For any $n\in\Z$, let $F_n=\{x:\, S_bf(x)>2^n\}$, and
\begin{displaymath}
\tau_n(x)=\{(p,q):\,\forall t\leq (p,q),E_t(\chi_{F_n})\leq 1/2\},
\end{displaymath}
where $E_t$ is the classical expectation operator. It is easy to check that $\tau_n$ is a stopping time, and $\tau_n\subset\tau_{n+1}$.

For each $n$, define a new function $f^{\tau_n}(x)=\sum_{t\in\tau_n(x)}\Delta^b_{t-1}f(x)$, then
\begin{displaymath}
f^{\tau_{n+1}}-f^{\tau_n}=\sum_{t\in\tau_{n+1}\setminus\tau_n}\Delta^b_{t-1}f.
\end{displaymath}
Using this, define
\begin{displaymath}
a^n=\frac{f^{\tau_{n+1}}-f^{\tau_n}}{2^n |\{x:\,\tau_n\neq\Z\times\Z\}|},\qquad \lambda_n=2^n|\{x:\,\tau_n\neq\Z\times\Z\}|. 
\end{displaymath}
We claim that such $a^n$ and $\lambda_n$ satisfy all the properties required in the theorem.

To check property $(2)$:
\begin{displaymath}
\begin{split}
\sum_n|\lambda_n|&=\sum_n2^n|\{x:\,\tau_n\neq\Z\times\Z\}|=\sum_n2^n|\{x:\,\sup_t|E_t(\chi_{F_n})|>1/2\}|\\
&\leq\sum_n2^n\cdot 4\int(\sup_t|E_t(\chi_{F_n})|)^2\\
&=4\sum_n2^n\|(\chi_{F_n})^*\|^2_{L^2}\\
&\lesssim 4\sum_n2^n\|\chi_{F_n}\|^2_{L^2}\\
&=4\sum_n2^n|\{x:\,S_bf(x)>2^n\}|\lesssim\|S_bf\|_{L^1}.
\end{split}
\end{displaymath}
In the above, the second line follows from Chebyshev Inequality, and the fourth line uses the $L^2$ boundedness of the classical martingale maximal function.

To check property $(1)$:
It suffices to check that
\begin{displaymath}
f=\sum_n(f^{\tau_{n+1}}-f^{\tau_n})=\lim_{n\rightarrow \infty}f^{\tau_n}-\lim_{n\rightarrow-\infty}f^{\tau_n},\,a.e.
\end{displaymath}
For the first limit, Chebyshev Inequality implies that 
\begin{displaymath}
|F_n|\leq2^{-n}\int S_bf=2^{-n}\|S_bf\|_{L^1}.
\end{displaymath}
So as $n\rightarrow\infty$, $|F_n|\rightarrow 0$ monotonically. Hence, $\|\sup_t|E_t(\chi_{F_n})|\|^2_{L^2}\approx\|\chi_{F_n}\|^2_{L^2}\rightarrow 0$.
By Fatou's Lemma,
\begin{displaymath}
\int \liminf_{n\rightarrow\infty} |(\chi_{F_n})^*|^2\leq\liminf_{n\rightarrow\infty}\int|(\chi_{F_n})^*|^2=0,
\end{displaymath}
which implies $\lim_{n\rightarrow\infty}|E_t(\chi_{F_n})|=0\,a.e.$ uniformly in $t$. So when $n$ is large enough, $\tau_n=\Z\times\Z\,a.e.$, i.e. $f^{\tau_n}=f$.

For the second limit, if $x$ is such that $S_bf(x)=0$, then $\Delta^b_tf(x)=0,\,\forall t$. Hence, $f^{\tau_n}(x)=0\,\,\forall n$.
Also, in this case, $\forall q$ fixed, since $f\in L^2(\R^n\times\R^m)$, 
\begin{displaymath}
\lim_{p\rightarrow-\infty}|E^{b_1}_pE^{b_2}_qf(x)|\lesssim\lim_{|I|\rightarrow\infty}\frac{\int_{I\times J}|f|}{|I\times J|}\leq\lim_{|I|\rightarrow\infty}\frac{\|f\|_{L^2}}{|I\times J|^{1/2}}=0.
\end{displaymath}
And similarly for $E^{b_1}_pE^{b_2}_{q+1}f(x)$. So $E^{b_1}_p\Delta^{b_2}_qf(x)=E^{b_1}_{p-1}\Delta^{b_2}_qf(x)=\lim_{p\rightarrow-\infty}E^{b_1}_p\Delta^{b_2}_qf(x)=0$, which means
\begin{displaymath}
E^{b_1}_pE^{b_2}_qf(x)=E^{b_1}_pE^{b_2}_{q+1}f(x),\,\forall p,q.
\end{displaymath} 
A similar limiting argument for the other variable implies $E^{b_1}_pE^{b_2}_qf(x)=0\,\forall p,q$. Hence, 
\begin{displaymath}
f(x)=\lim_{p,q\rightarrow\infty}f_{p,q}(x)=0.
\end{displaymath} 
Then the convergence is automatically true.

If $x$ is such that $S_bf(x)>0$, then for small enough $n$, $S_bf(x)>2^n$, i.e. $x\in F_n\subset X=\bigcup_n F_n=\{x:\,S_bf(x)>0\}$. Also,
\begin{displaymath}
\lim_{n\rightarrow-\infty}f^{\tau_n}(x)=\lim_{n\rightarrow-\infty}\sum_{t\in\tau_n}\Delta^b_{t-1}f(x)=\sum_{t\in\bigcap\tau_n}\Delta^b_{t-1}f(x).
\end{displaymath}
We claim that all the terms appearing in the sum are $0$, hence $\lim_{n\rightarrow-\infty}f^{\tau_n}(x)=0$.
 
For any $t\in\bigcap\tau_n(x)$, we have $E_t(\chi_{F_n})(x)\leq 1/2,\,\forall n$. Let $n\rightarrow-\infty$,
\begin{displaymath}
E_t(\chi_X)(x)=\lim_{n\rightarrow-\infty}E_t(\chi_{F_n})(x)\leq 1/2.
\end{displaymath}
Say $R=I\times J$ of size $2^{-t_1}\times 2^{-t_2}$ is the rectangle containing $x$ of generation $t$. Then $R$ has nonempty intersection with $X^c$ since otherwise $E_t(\chi_X)(x)=1$. For any $y\in R\cap X^c$, since $S_bf(y)=0$,
we have $\Delta^b_{p,q}f(y)=0,\,\forall p,q$. 

However, since $\Delta^b_{t-1}f(x)=\Delta^b_{t-1}f(y)$, it implies $\Delta^b_{t-1}f(x)=0$, which proves the claim. 

Then the only thing left to check is that all the $a^n$ defined are indeed atoms. 

To see this, firstly, $a^n\in L^2$. Indeed,
\begin{displaymath}
\|f^{\tau_{n+1}}-f^{\tau_n}\|^2_{L^2}\approx\int |S_b(f^{\tau_{n+1}}-f^{\tau_n})|^2=\int\sum_{t\in\tau_{n+1}\setminus\tau_n}|\Delta^b_{t-1}f|^2\leq\|S_bf\|^2_{L^2}\approx\|f\|^2_{L^2}.
\end{displaymath}

Secondly, just as how we argued for the second property above, we see that $|\{x:\,\tau_n(x)\neq\Z\times\Z\}|<\infty$.

Thirdly, if $t+1\in\tau_n$, for any double integer $s$ not satisfying $s\leq t$, by a simple computation, we have
\begin{displaymath}
E^b_t\Delta^b_{s-1}f=E^b_t(E^{b_1}_{s_1}E^{b_2}_{s_2}-E^{b_1}_{s_1-1}E^{b_2}_{s_2}-E^{b_1}_{s_1}E^{b_2}_{s_2-1}+E^b_{s-1})f=0.
\end{displaymath}
On the other hand, if $s\leq t$, then $s\in\tau_n$, hence,
\begin{displaymath}
E^b_t(\sum_{s\leq t,s\in\tau_{n+1}\setminus\tau_n}\Delta^b_{s-1}f)=0,
\end{displaymath}
which implies $a^n_t=E^b_t(a^n)=0$.

Finally, to show $\|(a^n)^*_b\|^2_{L^2}\lesssim |\{x:\,\tau_n(x)\neq\Z\times\Z\}|^{-1}$, it suffices to show
\begin{displaymath}
\|S_ba^n\|^2_{L^2}\lesssim |\{x:\,\tau_n(x)\neq\Z\times\Z\}|^{-1},
\end{displaymath}
which is equivalent to
\begin{displaymath}
\int\sum_{t\in\tau_{n+1}\setminus\tau_n}|\Delta^b_{t-1}f|^2\lesssim 4^n|\{x:\,\tau_n(x)\neq\Z\times\Z\}|.
\end{displaymath}
Write
\begin{displaymath}
LHS=\int\sum_{t\in\tau_{n+1}\setminus\tau_n}|\Delta^b_{t-1}f|^2\chi_{\{S_bf\leq 2^{n+1}\}}+\int\sum_{t\in\tau_{n+1}\setminus\tau_n}|\Delta^b_{t-1}f|^2\chi_{\{S_bf> 2^{n+1}\}}=I+II.
\end{displaymath} 
The first term can be dealt with trivially,
\begin{displaymath}
I\leq 4^{n+1}\int_{\mbox{spt}(S_ba^n)}\chi_{\{S_bf\leq 2^{n+1}\}}\leq 4^{n+1}|\{x:\,\tau_n(x)\neq\Z\times\Z\}|.
\end{displaymath}
For the second term, let $\mathcal{D}_t$ denote all those dyadic rectangles of generation $t$, then
\begin{displaymath}
\begin{split}
II&=\sum_{t\in\tau_{n+1}\setminus\tau_n}\sum_{R\in\mathcal{D}_t}\int_R |\Delta^b_{t-1}f|^2\chi_{\{S_bf>2^{n+1}\}}\\
&=\sum_{t\in\tau_{n+1}\setminus\tau_n}\sum_{R\in\mathcal{D}_t}|\Delta^b_{t-1}f|^2\mid_R |R|E_t(\chi_{F_{n+1}})\chi_R\\
&\leq \frac{1}{2}\sum_{t\in\tau_{n+1}\setminus\tau_n}\sum_{R\in\mathcal{D}_t}\int_R |\Delta^b_{t-1}f|^2\\
&=\frac{1}{2}\int\sum_{t\in\tau_{n+1}\setminus\tau_n}|\Delta^b_{t-1}f|^2
\end{split}
\end{displaymath}
In the above, the second lines follows from the fact that $\Delta^b_{t-1}f$ is a constant on each $R$, and the third line uses $t\in\tau_{n+1}$.
Combining $I$ and $II$ gives us
\begin{displaymath}
\int\sum_{t\in\tau_{n+1}\setminus\tau_n}|\Delta^b_{t-1}f|^2\lesssim 2\cdot 4^{n+1}|\{x:\,\tau_n(x)\neq\Z\times\Z\}|,
\end{displaymath}
which completes our proof for the theorem of atomic decomposition.
\end{proof}

With the result of Theorem \ref{squaremax}, we return to the full paraproducts, and give a proof of Proposition \ref{fullpara}.

\begin{proof}(of Proposition \ref{fullpara})
For any $f, g\in L^2(\R^n\times\R^m)$,
\begin{displaymath}
\begin{split}
|\pair{\pi^{b',b}_a(f)}{g}|&=|\pair{\sum_{K,V}\ave{f}^{b'}_{K\times V}M_b\Delta^{b_1}_K\Delta^{b_2}_Va}{g}|\\
&=|\pair{a}{M_b\sum_{K,V}\ave{f}^{b'}_{K\times V}\Delta^{b_1}_K\Delta^{b_2}_Vg}|\\
&\leq \|a\|_{BMO}\|M_b\sum_{K,V}\ave{f}^{b'}_{K\times V}\Delta^{b_1}_K\Delta^{b_2}_Vg\|_{H^1}=\|a\|_{BMO}\|\sum_{K,V}\ave{f}^{b'}_{K\times V}\Delta^{b_1}_K\Delta^{b_2}_Vg\|_{H^1_b}\\
&\lesssim \|a\|_{BMO}\|S_b(\sum_{K,V}\ave{f}^{b'}_{K\times V}\Delta^{b_1}_K\Delta^{b_2}_Vg)\|_{L^1},
\end{split}
\end{displaymath}
where the last step in the above follows from Theorem \ref{squaremax}. Hence, it suffices to show that
\begin{displaymath}
\|S_b(\sum_{K,V}\ave{f}^{b'}_{K\times V}\Delta^{b_1}_K\Delta^{b_2}_Vg)\|_{L^1}\lesssim\|f\|_{L^2}\|g\|_{L^2}.
\end{displaymath}

To see this, notice that
\begin{displaymath}
\begin{split}
S_b^2(\sum_{K,V}\ave{f}^{b'}_{K\times V}\Delta^{b_1}_K\Delta^{b_2}_Vg)&=\sum_{K,V}|\ave{f}^{b'}_{K\times V}\Delta^{b_1}_K\Delta^{b_2}_Vg|^2\chi_{K}\otimes \chi_{V}\\
&\leq |\sup_{K,V}\ave{f}^{b'}_{K\times V}\chi_K\otimes\chi_V|^2\sum_{K,V}|\Delta^{b_1}_K\Delta^{b_2}_Vg|^2\\
&\lesssim |M^S(f)|^2S_b^2(g),
\end{split}
\end{displaymath}
where $M^S(f)$ is the strong maximal function, which is bounded on $L^2$. Since $S_b$ is also bounded on $L^2$, we have
\begin{displaymath}
\begin{split}
&\|S_b(\sum_{K,V}\ave{f}^{b'}_{K\times V}\Delta^{b_1}_K\Delta^{b_2}_Vg)\|_{L^1}\\
&\lesssim \|M^S(f)S_b(g)\|_{L^1}\leq \|M^S(f)\|_{L^2}\|S_b(g)\|_{L^2}\lesssim \|f\|_{L^2}\|g\|_{L^2}.
\end{split}
\end{displaymath}
\end{proof}

\subsection*{Mixed paraproducts}
Since we are working in the bi-parameter setting, there appears a new mixed type of $b$-adapted paraproducts which requires particular attention. Basically, it means we have an average on $a$, and a difference on $f$ with respect to one variable, and conversely with respect to the other. 

\begin{defn}
For $a\in BMO(\R^n\times\R^m)$, operator $\tilde{\pi}^{b',b}_a$ is called a mixed paraproduct, defined as
\begin{displaymath}
\tilde{\pi}^{b',b}_a(f)=\sum_{K\in\mathcal{D}^n,V\in\mathcal{D}^m}E^{b'_1*}_K((E^{b'_2}_Vf)M_b\Delta^{b_1}_K\Delta^{b_2}_Va).
\end{displaymath}
\end{defn}

\begin{prop}
Mixed paraproducts are bounded operators on $L^2(\R^n\times\R^m)$. Specifically,
\begin{displaymath}
\|\tilde{\pi}^{b',b}_a(f)\|_{L^2(\R^n\times\R^m)}\lesssim\|a\|_{BMO(\R^n\times\R^m)}\|f\|_{L^2(\R^n\times\R^m)}.
\end{displaymath}
\end{prop}

Since we already have the $b$-adapted square function characterization of $H^1_b$, this proposition can be proved in the same way as a similar result in \cite{PV}.

\begin{proof}
For any $f,g\in L^2(\R^n\times\R^m)$,
\begin{displaymath}
\begin{split}
|\pair{\tilde{\pi}^{b',b}_a(f)}{g}|&=|\sum_{K,V}\pair{E^{b'_1*}_K(E^{b'_2}_Vf)M_b\Delta^{b_1}_K\Delta^{b_2}_Va}{g}|\\
&=|\sum_{K,V}\pair{a}{M_b(\Delta^{b_1}_KE^{b'_2}_Vf)(\Delta^{b_2}_VE^{b'_1}_Kg)}|\\
&=|\pair{a}{\sum_{K,V}M_b\ave{\Delta^{b_1}_Kf}^{b'_2}_V\otimes\ave{\Delta^{b_2}_Vg}^{b'_1}_K}|\\
&\leq \|a\|_{BMO}\|\sum_{K,V}M_b\ave{\Delta^{b_1}_Kf}^{b'_2}_V\otimes\ave{\Delta^{b_2}_Vg}^{b'_1}_K\|_{H^1}\\
&=\|a\|_{BMO}\|\sum_{K,V}\ave{\Delta^{b_1}_Kf}^{b'_2}_V\otimes\ave{\Delta^{b_2}_Vg}^{b'_1}_K\|_{H^1_b}\\
&\lesssim \|a\|_{BMO}\|S_b(\sum_{K,V}\ave{\Delta^{b_1}_Kf}^{b'_2}_V\otimes\ave{\Delta^{b_2}_Vg}^{b'_1}_K)\|_{L^1}.
\end{split}
\end{displaymath}

We claim that
\begin{displaymath}
\|S_b(\sum_{K,V}\ave{\Delta^{b_1}_Kf}^{b'_2}_V\otimes\ave{\Delta^{b_2}_Vg}^{b'_1}_K)\|_{L^1}\lesssim\|f\|_{L^2}\|g\|_{L^2}.
\end{displaymath}

To see this, note that
\begin{displaymath}
\begin{split}
&S^2_b(\sum_{K,V}\ave{\Delta^{b_1}_Kf}^{b'_2}_V\otimes\ave{\Delta^{b_2}_Vg}^{b'_1}_K)\\
&=\sum_{K,V}|\Delta^{b_1}_K(\ave{f}^{b'_2}_V)\otimes\Delta^{b_2}_V(\ave{g}^{b'_1}_K)|^2\chi_K\otimes\chi_V\\
&\leq(\sum_K\sup_V|\Delta^{b_1}_K(\ave{f}^{b'_2}_V)|^2\chi_K\otimes\chi_V)\cdot (\sum_V|\Delta^{b_2}_V(\ave{g}^{b'_1}_K)|^2\chi_K\otimes\chi_V)\\
&\leq (\sum_K\sup_V|\Delta^{b_1}_K(\ave{f}^{b'_2}_V)|^2\chi_K\otimes\chi_V)\cdot (\sup_K\sum_V|\Delta^{b_2}_V(\ave{g}^{b'_1}_K)|^2\chi_K\otimes\chi_V)\\
&:= |S_{b_1}M^{b'_2}(f)|^2|M^{b'_1}S_{b_2}(g)|^2,
\end{split}
\end{displaymath}
where the last two operators are just formally defined, but not the compositions of the square functions and maximal functions. Since pointwisely, $|M^{b'_1}(S_{b_2}g)|\leq|S_{b_2}(M^{b'_1}g)|$, by symmetry, it suffices to prove that $S_{b_1}M^{b'_2}:\,L^2\rightarrow L^2$. And this is true because
\begin{displaymath}
\begin{split}
\|S_{b_1}M^{b'_2}f\|_{L^2}&=\|(\sum_K\sup_V|\ave{\Delta^{b_1}_Kf}^{b'_2}_V|^2\chi_K\otimes\chi_V)^{1/2}\|_{L^2}\\
&\lesssim \|(\sum_K|M_2(\Delta^{b_1}_Kf)|^2\chi_K)^{1/2}\|_{L^2}\\
&\lesssim \sum_{i=1}^{2^n}(\int_{\R^n}\|(\sum_K\chi_{K_i}(x)\otimes M_2(\Delta^{b_1}_Kf)|_{K_i}^2(y))^{1/2}\|_{L^2(\R^m)}^2\,dx)^{1/2}\\
&\lesssim \sum_{i=1}^{2^n}(\int_{\R^n}\|(\sum_K|\Delta^{b_1}_Kf|^2\chi_{K_i})^{1/2}\|_{L^2(\R^m)}^2\,dx)^{1/2}\\
&\lesssim (\int_{\R^n}\|(\sum_K|\Delta^{b_1}_Kf|^2)^{1/2}\|_{L^2(\R^m)}^2\,dx)^{1/2}\\
&=(\int_{\R^m}\|S_{b_1}f_y\|_{L^2(\R^n)}^2\,dy)^{1/2}\\
&\approx (\int_{\R^m}\|f_y\|^2_{L^2(\R^n)}\,dy)^{1/2}=\|f\|_{L^2}.
\end{split}
\end{displaymath}
In the above, $M_2$ means the Hardy-Littlewood maximal function with respect to the second variable. In the fourth line, we used the Fefferman-Stein inequality. And in the sixth line, the operator $S_{b_1}$ is the one-parameter $b_1$-adapted square function, defined as $S_{b_1}f=(\sum_I |\Delta^{b_1}_I f|^2)^{1/2}$. It is straightforward to see that $S_{b_1}$ is an $L^2$ isometry up to some constant, which implies the seventh line in the above, where $f_y(x)$ denotes $f(x,y)$.

Hence, the $L^2$ boundedness of the mixed paraproduct is fully justified.
\end{proof}

\section{Main theorem and the strategy}

We return to the main theorem of this paper. We will prove that, under the assumptions stated in Section \ref{assump}, $T$ is bounded on $L^2(\R^n\times\R^m)$ with the operator norm depending only on the constants appearing in the above weak assumptions. By density and boundedness of $b,b'$, it suffices to show that for any $C_0^\infty$ functions $f,g$, there is a universal constant $C$ such that
\begin{displaymath}
|\pair{M_{b'}TM_bf}{g}|\leq C\|f\|_{L^2}\|g\|_{L^2}.
\end{displaymath} 

To prove this, recall that Martikainen \cite{Ma} gave an averaging formula for the bilinear form $\pair{Tf}{g}$ using a probabilistic concept called "goodness" of cubes. Here, if we decompose $f$ using the new defined $b$-adapted martingale difference instead, there is a natural generalization of the averaging formula as follows.

\begin{prop}\label{avefor}
\begin{displaymath}
\begin{split}
\pair{M_{b'}TM_bf}{g}&=\frac{1}{\pi^n_{\good}\pi^m_{\good}}\mathbb{E}_{\omega^n}\mathbb{E}_{\omega^m}\cdot\\
&\sum_{\substack{I_1,I_2\in\mathcal{D}^n\\J_1,J_2\in\mathcal{D}^m}}\chi_{\good}(\operatorname{small}(I_1,I_2))\chi_{\good}(\operatorname{small}(J_1,J_2))\pair{M_{b'}TM_b\Delta^{b_1}_{I_1}\Delta^{b_2}_{J_1}f}{\Delta^{b'_1}_{I_2}\Delta^{b'_2}_{J_2}g}.
\end{split}
\end{displaymath}
\end{prop} 

To understand the above formula, recall that in \cite{Hy}, a cube $I \in \mathcal{D}^n_{\omega^n}$ is called bad if there exists $\tilde I \in \mathcal{D}^n_{\omega^n}$ so that $\ell(\tilde I) \ge 2^r \ell(I)$ and $d(I, \partial \tilde I) \le 2\ell(I)^{\gamma_n}\ell(\tilde I)^{1-\gamma_n}$. $\gamma_n = \delta/(2n + 2\delta)$, where $\delta > 0$ appears in the kernel estimates. And
$\pi^n_{\textrm{good}}:= \mathbb{P}_{\omega^n}(I \dotplus\omega^n \textrm{ is good})$ is independent of $I \in \mathcal{D}^n_0$. By lemma 2.3 in \cite{Hy}, the parameter $r$ can be chosen large enough such that $\pi^n_{\textrm{good}} > 0$.
Moreover, for a fixed $I \in \mathcal{D}^n_0$
the position of $I \dotplus \omega^n$ depends on $\omega^n_i$ with $2^{-i} < \ell(I)$, while the goodness of $I \dotplus \omega^n$ depends on $\omega^n_i$ with $2^{-i} \ge \ell(I)$. Hence, they are independent. The proof of Proposition \ref{avefor} is identical to the proof of Proposition 2.1 in \cite{Ma}, which we omit here.

Note that as in \cite{Hy} and \cite{Ma}, we do need to justify that the sum on the right hand side converges to the left hand side, which is the only place throughout the paper where we use the a priori $L^2\rightarrow L^2$ boundedness of $T$. Indeed, by the convergence of expectation operators in $L^2$, the boundedness of $T$ will easily imply the convergences in the formula. However, when dealing with specific operators in practice, sometimes we can prove the convergence of the formula without assuming the boundedness assumption.

For example, if $T$ is canonically associated with a standard antisymmetric kernel $K(x,y)$, in the sense that
\begin{displaymath}
K(x,y)=-K(y_1,x_2,x_1,y_2)=-K(x_1,y_2,y_1,x_2),
\end{displaymath}
and $K$ satisfies all the size and H\"older conditions.

Then for any $f=f_1\otimes f_2,g=g_1\otimes g_2\in C^\infty_0(\R^n\times\R^m)$, 
\begin{displaymath}
\pair{Tf}{g}=\int\int\int\int K(x,y)f(y)g(x)\,dxdy
\end{displaymath}
is well defined. Hence, we automatically have the full and partial kernel representations. Also, by antisymmetry,
\begin{displaymath}
\pair{T(\chi_K\otimes\chi_V)}{\chi_K\otimes\chi_V}=0,
\end{displaymath}
which corresponds to the weak boundedness property for $b=b'=1$. With these observations in mind, it is not hard to show that for any $f,g\in C^\infty_0$ and any fixed dyadic grid,
\begin{displaymath}
\pair{Tf}{g}=\lim_{N\rightarrow\infty}\sum_{|p_i|,|q_i|<N}\pair{T\Delta_{p_1}\Delta_{q_1}f}{\Delta_{p_2}\Delta_{q_2}g}.
\end{displaymath}
So the a priori boundedness of $T$ is not necessary any more.

With the averaging formula, it suffices to bound the sum on the right hand side uniformly for any fixed random grids, to do which, we will divide the sum into different parts according to the relative positions of the cubes, and discuss different cases one by one. By symmetry, except for one mixed case ($\ell(I_1)\leq\ell(I_2),\ell(J_1)>\ell(J_2)$), all the other cases are symmetric to $(\ell(I_1)\leq\ell(I_2), \ell(J_1)\leq\ell(J_2))$, which we will start with. 

For the relative position of $I_1, I_2$, there are four different cases: separated (i.e. $\dist(I_1,I_2)>\ell(I_1)^{\gamma_n}\ell(I_2)^{1-\gamma_n}$), inside (i.e. $I_1\subsetneq I_2$), equal, nearby (i.e. $\dist(I_1,I_2)\leq\ell(I_1)^{\gamma_n}\ell(I_2)^{1-\gamma_n}$). Similarly, there are also four different cases for the second variable. Again using symmetry, it suffices to analyze the following ten cases: 
\begin{itemize}
\item separated/separated, separated/inside, separated/equal, separated/nearby, \\
\item inside/inside, inside/equal, inside/nearby, \\
\item equal/equal, equal/nearby, nearby/nearby. 
\end{itemize}

In preparation, we state two control lemma here which will be repeatedly used when we deal with different cases in the following. For simplicity of notation, write
\begin{displaymath}
\sum_{I_1,I_2\subset K}^{(i_1,i_2)}=\sum_{\substack{I_1,I_2\in\mathcal{D}^n(K)\\ \ell(I_1)=2^{-i_1}\ell(K)\\ \ell(I_2)=2^{-i_2}\ell(K)}},
\end{displaymath}
where $K\in\mathcal{D}^n$ and $i_1,i_2\in\mathbb{N}$.

\begin{lem}{(Full control lemma)}
For fixed $i_1,i_2,j_1,j_2\in\N$ and any $f\in L^2(\R^n\times\R^m)$, $g\in L^2(\R^n\times\R^m)$,
\begin{displaymath}
\sum_{\substack{K\in\mathcal{D}^n\\V\in\mathcal{D}^m}}\sum_{I_1,I_2\subset K}^{(i_1,i_2)}\sum_{J_1,J_2\subset V}^{(j_1,j_2)}\frac{|I_1|^{1/2}|I_2|^{1/2}}{|K|}\frac{|J_1|^{1/2}|J_2|^{1/2}}{|V|}\|\Delta^{b_1}_{I_1}\Delta^{b_2}_{J_1}f\|_{L^2}\|\Delta^{b'_1}_{I_2}\Delta^{b'_2}_{J_2}g\|_{L^2}\lesssim \|f\|_{L^2}\|g\|_{L^2}.
\end{displaymath}
\end{lem}

\begin{proof}
It follows as a consequence of H\"{o}lder inequality.
\begin{displaymath}
\begin{split}
LHS&\leq (\sum_{K,V}\sum_{I_1,I_2\subset K}^{(i_1,i_2)}\sum_{J_1,J_2\subset V}^{(j_1,j_2)}\frac{|I_1||J_1|}{|K||V|}\|\Delta^{b_1}_{I_1}\Delta^{b_2}_{J_1}f\|^2_{L^2})^{\frac{1}{2}}(\sum_{K,V}\sum_{I_1,I_2\subset K}^{(i_1,i_2)}\sum_{J_1,J_2\subset V}^{(j_1,j_2)}\frac{|I_2||J_2|}{|K||V|}\|\Delta^{b'_1}_{I_2}\Delta^{b'_2}_{J_2}g\|^2_{L^2})^{\frac{1}{2}}\\
&=(2^{i_2n}2^{j_2m}\sum_{K,V}\sum_{I_1\subset K}^{(i_1)}\sum_{J_1\subset V}^{(j_1)}2^{-i_1n}2^{-j_1m}\|\Delta^{b_1}_{I_1}\Delta^{b_2}_{J_1}f\|^2_{L^2})^{\frac{1}{2}}\cdot\\
&\qquad(2^{i_1n}2^{j_1m}\sum_{K,V}\sum_{I_2\subset K}^{(i_2)}\sum_{J_2\subset V}^{(j_2)}2^{-i_2n}2^{-j_2m}\|\Delta^{b'_1}_{I_2}\Delta^{b'_2}_{J_2}f\|^2_{L^2})^{\frac{1}{2}}\\
&=(\sum_{K,V}\sum_{I_1\subset K}^{(i_1)}\sum_{J_1\subset V}^{(j_1)}\|\Delta^{b_1}_{I_1}\Delta^{b_2}_{J_1}f\|^2_{L^2})^{\frac{1}{2}}(\sum_{K,V}\sum_{I_2\subset K}^{(i_2)}\sum_{J_2\subset V}^{(j_2)}\|\Delta^{b'1_1}_{I_2}\Delta^{b'_2}_{J_2}g\|^2_{L^2})^{\frac{1}{2}}\\
&\lesssim \|f\|_{L^2}\|g\|_{L^2}.
\end{split}
\end{displaymath}
In the last step above, we used the $L^2$ property of $b$-adapted double martingale difference.
\end{proof}

\begin{lem}{(Partial control lemma)}
For fixed $i_1,i_2,j_1,j_2\in\N$ and any $f\in L^2(\R^n\times\R^m)$, $g\in L^2(\R^n\times\R^m)$,
\begin{displaymath}
\sum_{K\in\mathcal{D}^n}\sum_{I_1,I_2\subset K}^{(i_1,i_2)}\frac{|I_1|^{1/2}|I_2|^{1/2}}{|K|}\|\Delta^{b_1}_{I_1}f\|_{L^2}\|\Delta^{b'_1}_{I_2}g\|_{L^2}\lesssim \|f\|_{L^2}\|g\|_{L^2},
\end{displaymath}
and
\begin{displaymath}
\sum_{V\in\mathcal{D}^m}\sum_{J_1,J_2\subset V}^{(j_1,j_2)}\frac{|J_1|^{1/2}|J_2|^{1/2}}{|V|}\|\Delta^{b_2}_{J_1}f\|_{L^2}\|\Delta^{b'_2}_{J_2}g\|_{L^2}\lesssim \|f\|_{L^2}\|g\|_{L^2}.
\end{displaymath}
\end{lem}

These two inequalities are symmetric, and they can both be derived using a similar technique as for the above lemma. The only difference here is that we need to use the $L^2$ property of the $b$-adapted martingale difference of only one variable instead.

Before we move on to the main part of the proof of the theorem, i.e. the case by case estimate of summands in the averaging formula, let's look at an example to see how our theory fits into some known results of boundedness of bi-parameter singular integral operators.

Consider operators associated with antisymmetric standard kernels. Journ\'e, in \cite{Jo}, proved that if $K=L\tilde{A}$, the bicommutator of Calder\'on-Coifman type, where $L$ is any standard antisymmetric function, and 
\begin{displaymath}
\tilde{A}(x,y)=\frac{A(x_1,x_2)+A(y_1,y_2)-A(y_1,x_2)-A(x_1,y_2)}{(x_1-y_1)(x_2-y_2)},
\end{displaymath}
for some $A:\,\R^n\times\R^m\rightarrow \C$ such that $\partial^2_{12}A\in L^\infty$, then, the $L^2\rightarrow L^2$ boundedness of the operator associated to $L$ implies $T1\in BMO$, as well as the other $BMO$ conditions. It is also not hard to verify directly that $T$ satisfies the weak boundedness property and the four diagonal $BMO$ assumptions. (All of them are actually zero!). Hence, by our main theorem, $T$ is bounded on $L^2$ with operator norm controlled by the weak assumptions.

\section{Separated/Separated: $\sigma_{\out/\out}$}
Define $I_1\vee I_2 = \bigcap_{K \in \mathcal{D}^n, \, K \supset I_1 \cup I_2} K$, i.e. the smallest $K$ such that $I_1\cup I_2\subset K$, and similarly for $J_1\vee J_2$. Then since both of them are separated and $I_1,J_1$ are good, it is proved in \cite{Hy} by Hyt\"{o}nen that $\ell(I_1)^{\gamma_n}\ell(K)^{1-\gamma_n} \lesssim \dist(I_1,I_2)$ and $\ell(J_1)^{\gamma_m}\ell(V)^{1-\gamma_m}\lesssim \dist(J_1,J_2)$.

Hence, we can write

\begin{displaymath}
\sigma_{\out/\out}=\sum_{i_2=1}^\infty\sum_{i_1=i_2}^\infty\sum_{j_2=1}^\infty\sum_{j_1=j_2}^\infty\sum_{K,V}\sum_{\substack{I_1,I_2\subset K\\I_1\vee I_2=K\\I_1,I_2\,\mbox{separated}}}^{(i_1,i_2)}\sum_{\substack{J_1,J_2\subset V\\J_1\vee J_2=V\\J_1,J_2\,\mbox{separated}}}^{(j_1,j_2)}\pair{M_{b'}TM_b\Delta^{b_1}_{I_1}\Delta^{b_2}_{J_1}f}{\Delta^{b'_1}_{I_2}\Delta^{b'_2}_{J_2}g}.
\end{displaymath}

The main goal of this section is to show that the following inequality holds.

\begin{prop}
\begin{displaymath}
\begin{split}
&|\pair{M_{b'}TM_b\Delta^{b_1}_{I_1}\Delta^{b_2}_{J_1}f}{\Delta^{b'_1}_{I_2}\Delta^{b'_2}_{J_2}g}|\\
&\lesssim 2^{-i_1\delta/2}2^{-j_1\delta/2}\frac{|I_1|^{1/2}|I_2|^{1/2}}{|K|}\frac{|J_1|^{1/2}|J_2|^{1/2}}{|V|}\|\Delta^{b_1}_{I_1}\Delta^{b_2}_{J_1}f\|_{L^2}\|\Delta^{b'_1}_{I_2}\Delta^{b'_2}_{J_2}g\|_{L^2}.
\end{split}
\end{displaymath}
\end{prop}

If this is true, then by the full control lemma we stated in the beginning, $\sigma_{\out/\out}$ can be bounded by $\|f\|_{L^2}\|g\|_{L^2}$. 

\begin{proof}
Since the two functions are well separated on both variables, by the full kernel representation,
\begin{displaymath}
LHS=|\int_{I_1\times J_1}\int_{I_2\times J_2} K(x,y)\Delta^{b_1}_{I_1}\Delta^{b_2}_{J_1}f(y)b(y)\Delta^{b'_1}_{I_2}\Delta^{b'_2}_{J_2}g(x)b'(x)\,dxdy|.
\end{displaymath}

Using the cancellation properties of the martingale differences, we can replace $K(x,y)$ in the above by
\begin{equation}\label{sepsepker} 
K(x,y)-K(x,y_1,c_{J_1})-K(x,c_{I_1},y_2)+K(x,c_{I_1},c_{J_1}).
\end{equation}
Since $|y_1-c_{I_1}|\leq \ell(I_1)/2\leq\frac{1}{2}\ell(I_1)^{\gamma_n}\ell(I_2)^{1-\gamma_n}\leq \dist(I_1,I_2)/2\leq |x_1-c_{I_1}|/2$, and similarly $|y_2-c_{J_1}|\leq |x_2-c_{J_1}|/2$, by the full H\"{o}lder condition,
\begin{displaymath}
\begin{split}
|(\ref{sepsepker})|&\lesssim \frac{|y_1-c_{I_1}|^\delta}{|x_1-c_{I_1}|^{n+\delta}}\frac{|y_2-c_{J_1}|^\delta}{|x_2-c_{J_1}|^{m+\delta}}\\
&\lesssim \ell(I_1)^\delta\dist(I_1,I_2)^{-n-\delta}\ell(J_1)^\delta\dist(J_1,J_2)^{-m-\delta}\\
&\lesssim \ell(I_1)^{\delta/2}\ell(K)^{-\delta/2}|K|^{-1}\ell(J_1)^{\delta/2}\ell(V)^{-\delta/2}|V|^{-1}\\
&=2^{-i_1\delta/2}2^{-j_1\delta/2}|K|^{-1}|V|^{-1},
\end{split}
\end{displaymath}
where for the third inequality we used $\ell(I_1)^{\gamma_n}\ell(K)^{1-\gamma_n} \lesssim \dist(I_1,I_2)$ and $\ell(J_1)^{\gamma_m}\ell(V)^{1-\gamma_m}\lesssim \dist(J_1,J_2)$. Then, by H\"{o}lder inequality and the boundedness of $b,b'$, this implies
\begin{displaymath}
\begin{split}
LHS&\lesssim 2^{-i_1\delta/2}2^{-j_1\delta/2}|K|^{-1}|V|^{-1}(\int_{I_1\times J_1}|\Delta^{b_1}_{I_1}\Delta^{b_2}_{J_1}f(y)|\,dy)(\int_{I_2\times J_2}|\Delta^{b'_1}_{I_2}\Delta^{b'_2}_{J_2}g(x)|\,dx)\\
&\leq RHS.
\end{split}
\end{displaymath}

\end{proof}

\section{Separated/Inside: $\sigma_{\out/\inside}$}

Since $J_1\subsetneq J_2$, $J_1$ is contained in some child of $J_2$, which we denote by $J_{2,1}$. Then $\Delta^{b'_1}_{I_2}\Delta^{b'_2}_{J_2}g$ is constant with respect to $x_2$ on $J_{2,1}$, and we have
\begin{displaymath}
\begin{split}
&\pair{M_{b'}TM_b\Delta^{b_1}_{I_1}\Delta^{b_2}_{J_1}f}{\Delta^{b'_1}_{I_2}\Delta^{b'_2}_{J_2}g}\\
&=\pair{M_{b'}TM_b\Delta^{b_1}_{I_1}\Delta^{b_2}_{J_1}f}{(\chi_{J_{2,1}}+\chi_{J_{2,1}^c})\Delta^{b'_1}_{I_2}\Delta^{b'_2}_{J_2}g}\\
&=\pair{M_{b'}TM_b\Delta^{b_1}_{I_1}\Delta^{b_2}_{J_1}f}{\chi_{J_{2,1}^c}(\Delta^{b'_1}_{I_2}\Delta^{b'_2}_{J_2}g-\ave{\Delta^{b'_1}_{I_2}\Delta^{b'_2}_{J_2}g}^{b'_2}_{J_{2,1}})}\\
&\qquad +\pair{M_{b'}TM_b\Delta^{b_1}_{I_1}\Delta^{b_2}_{J_1}f}{\ave{\Delta^{b'_1}_{I_2}\Delta^{b'_2}_{J_2}g}^{b'_2}_{J_{2,1}}(x_1)\otimes 1(x_2)}\\
&:=I+II
\end{split}
\end{displaymath}
where $\ave{f}^{b_2}_{J}$ denotes the $b_2$-adapted average of $f$ over $J$ with respect to the second variable: $(\int_{J}b_2)^{-1}(\int_{J}fb_2)$.

Write
\begin{displaymath}
\begin{split}
\sigma_{\out/\inside}&=\sum_{i_2=1}^\infty\sum_{i_1=i_2}^\infty\sum_{j_1=1}^\infty\sum_{K\in\mathcal{D}^n}\sum_{J_2\in\mathcal{D}^m}\sum_{\substack{\dist(I_1,I_2)>\ell(I_1)^{\gamma_n}\ell(I_2)^{1-\gamma_n}\\I_1\vee I_2=K}}^{(i_1,i_2)}\sum_{J_1\subset J_2}^{(j_1)} I+II\\
&:=\sigma_{\out/\inside}'+\sigma_{\out/\inside}''.
\end{split}
\end{displaymath}

\subsection*{Part $\sigma_{\out/\inside}'$}
In order to bound $\sigma_{\out/\inside}'$ by $\|f\|_{L^2}\|g\|_{L^2}$, by the full control lemma, it suffices to prove the following.
\begin{prop}
\begin{displaymath}
|I|\lesssim \frac{|I_1|^{1/2}|I_2|^{1/2}}{|K|}\frac{|J_1|^{1/2}}{|J_2|^{1/2}}2^{-i_1\delta/2}2^{-j_1\delta/2}\|\Delta^{b_1}_{I_1}\Delta^{b_2}_{J_1}f\|_{L^2}\|\Delta^{b'_1}_{I_2}\Delta^{b'_2}_{J_2}g\|_{L^2}.
\end{displaymath}
\end{prop}

\begin{proof}
Case 1: $\ell(J_1)<2^{-r}\ell(J_2)$.

The two functions in the pairing are separated in both variables, which enables us to use the full kernel representation:
\begin{displaymath}
I=\int_{I_1\times J_1}\int_{I_2\times J_{2,1}^c}K(x,y)\Delta^{b_1}_{I_1}\Delta^{b_2}_{J_1}f(y)b(y)(\Delta^{b'_1}_{I_2}\Delta^{b'_2}_{J_2}g(x)-\ave{\Delta^{b'_1}_{I_2}\Delta^{b'_2}_{J_2}g}^{b'_2}_{J_{2,1}})b'(x)\,dxdy.
\end{displaymath}
Since in this case, the size of $J_1$ is "significantly" small compared with $J_2$, by the goodness of $J_1$, $\dist(J_1, J_{2,1}^c)\geq 2\ell(J_1)^{\gamma_m}\ell(J_{2,1})^{1-\gamma_m}\geq \ell(J_1)^{\gamma_m}\ell(J_2)^{1-\gamma_m}$, which implies good separation on both variables. Hence, using the cancellation property in $y$ variable, we can change the kernel $K(x,y)$ in the above to
\begin{displaymath}
K(x,y)-K(x,y_1,c_{J_1})-K(x,c_{I_1},y_2)+K(x,c_{I_1},c_{J_1}).
\end{displaymath}
By H\"{o}lder condition and a similar computation as in the Separated/Separated case,
\begin{displaymath}
\begin{split}
|I|&\lesssim \ell(I_1)^{\delta/2}\ell(K)^{-\delta/2}|K|^{-1}\ell(J_1)^\delta(\int_{I_1\times J_1}|\Delta^{b_1}_{I_1}\Delta^{b_2}_{J_1}f|\,dy)\cdot\\
&\qquad (\int_{I_2\times J_{2,1}^c}\frac{1}{|x_2-c_{J_1}|^{m+\delta}}|\Delta^{b'_1}_{I_2}\Delta^{b'_2}_{J_2}g-\ave{\Delta^{b'_1}_{I_2}\Delta^{b'_2}_{J_2}g}^{b'_2}_{J_{2,1}}|\,dx)\\
&\leq 2^{-i_1\delta/2}|K|^{-1}\ell(J_1)^\delta\|\Delta^{b_1}_{I_1}\Delta^{b_2}_{J_1}f\|_{L^2}|I_1|^{1/2}|J_1|^{1/2}\cdot\\
&\qquad (\int_{I_2\times J_{2,1}^c}\frac{1}{|x_2-c_{J_1}|^{m+\delta}}(|\Delta^{b'_1}_{I_2}\Delta^{b'_2}_{J_2}g|+|\ave{\Delta^{b'_1}_{I_2}\Delta^{b'_2}_{J_2}g}^{b'_2}_{J_{2,1}}|)\,dx)\\
&=2^{-i_1\delta/2}\frac{|I_1|^{1/2}}{|K|}|J_1|^{1/2}\ell(J_1)^\delta\|\Delta^{b_1}_{I_1}\Delta^{b_2}_{J_1}f\|_{L^2}\cdot\\
&\qquad (\int_{I_2\times J_{2,1}^c}\frac{1}{|x_2-c_{J_1}|^{m+\delta}}|\ave{\Delta^{b'_1}_{I_2}\Delta^{b'_2}_{J_2}g}^{b'_2}_{J_{2,1}}|\,dx+\sum_{j=2}^{2^m}\int_{I_2\times J_{2,j}}\frac{1}{|x_2-c_{J_1}|^{m+\delta}}|\ave{\Delta^{b'_1}_{I_2}\Delta^{b'_2}_{J_2}g}^{b'_2}_{J_{2,j}}|\,dx)\\
&\lesssim 2^{-i_1\delta/2}\frac{|I_1|^{1/2}}{|K|}|J_1|^{1/2}\ell(J_1)^\delta\|\Delta^{b_1}_{I_1}\Delta^{b_2}_{J_1}f\|_{L^2}\|\Delta^{b'_1}_{I_2}\Delta^{b'_2}_{J_2}g\|_{L^2}|I_2|^{-1/2}|J_2|^{-1/2}\int_{I_2\times J_{2,1}^c}\frac{1}{|x_2-c_{J_1}|^{m+\delta}}\,dx\\
&\lesssim 2^{-i_1\delta/2}\frac{|I_1|^{1/2}|I_2|^{1/2}}{|K|}\frac{|J_1|^{1/2}}{|J_2|^{1/2}}\ell(J_1)^\delta\|\Delta^{b_1}_{I_1}\Delta^{b_2}_{J_1}f\|_{L^2}\|\Delta^{b'_1}_{I_2}\Delta^{b'_2}_{J_2}g\|_{L^2}\dist(J_1,J_{2,1}^c)^{-\delta}\\
&\leq 2^{-i_1\delta/2}\frac{|I_1|^{1/2}|I_2|^{1/2}}{|K|}\frac{|J_1|^{1/2}}{|J_2|^{1/2}}\frac{\ell(J_1)^{\delta/2}}{\ell(J_2)^{\delta/2}}\|\Delta^{b_1}_{I_1}\Delta^{b_2}_{J_1}f\|_{L^2}\|\Delta^{b'_1}_{I_2}\Delta^{b'_2}_{J_2}g\|_{L^2}\\
&=LHS,
\end{split}
\end{displaymath}
where in the third line, $J_{2,j}$ denotes all the children of $J_2$ except $J_{2,1}$, and we used the fact that $\Delta^{b'_1}_{I_2}\Delta^{b'_2}_{J_2}g$ is constant with respect to $x_2$ on each child of $J_2$. And the fourth line follows from the estimate of those averages of $\Delta^{b'_1}_{I_2}\Delta^{b'_2}_{J_2}g$.

Case 2: $2^{-r}\ell(J_2)\leq\ell(J_1)\leq\ell(J_2)$.

Let's further split $I$ into two parts:
\begin{displaymath}
I'=\pair{M_{b'}TM_b\Delta^{b_1}_{I_1}\Delta^{b_2}_{J_1}f}{\chi_{3J_1\cap J_{2,1}^c}(\Delta^{b'_1}_{I_2}\Delta^{b'_2}_{J_2}g-\ave{\Delta^{b'_1}_{I_2}\Delta^{b'_2}_{J_2}g}^{b'_2}_{J_{2,1}})},
\end{displaymath}
\begin{displaymath}
I''=\pair{M_{b'}TM_b\Delta^{b_1}_{I_1}\Delta^{b_2}_{J_1}f}{\chi_{(3J_1)^c\cap J_{2,1}^c}(\Delta^{b'_1}_{I_2}\Delta^{b'_2}_{J_2}g-\ave{\Delta^{b'_1}_{I_2}\Delta^{b'_2}_{J_2}g}^{b'_2}_{J_{2,1}})}.
\end{displaymath}

In $I''$, we still have good separation on both variables, so following from exact the same computation in Case Separated/Separated and the fact that now the size of $J_1, J_2$ are comparable,
\begin{displaymath}
\begin{split}
|I''|&\lesssim 2^{-i_1\delta/2}\frac{|I_1|^{1/2}|I_2|^{1/2}}{|K|}\frac{|J_1|^{1/2}}{|J_2|^{1/2}}\|\Delta^{b_1}_{I_1}\Delta^{b_2}_{J_1}f\|_{L^2}\|\Delta^{b'_1}_{I_2}\Delta^{b'_2}_{J_2}g\|_{L^2}\ell(J_1)^\delta\int_{(3J_1)^c}\frac{1}{|x_2-c_{J_1}|^{m+\delta}}\\
&\lesssim 2^{-i_1\delta/2}\frac{|I_1|^{1/2}|I_2|^{1/2}}{|K|}\frac{|J_1|^{1/2}}{|J_2|^{1/2}}\|\Delta^{b_1}_{I_1}\Delta^{b_2}_{J_1}f\|_{L^2}\|\Delta^{b'_1}_{I_2}\Delta^{b'_2}_{J_2}g\|_{L^2}\\
&\lesssim 2^{-i_1\delta/2}2^{-j_1\delta/2}\frac{|I_1|^{1/2}|I_2|^{1/2}}{|K|}\frac{|J_1|^{1/2}}{|J_2|^{1/2}}\|\Delta^{b_1}_{I_1}\Delta^{b_2}_{J_1}f\|_{L^2}\|\Delta^{b'_1}_{I_2}\Delta^{b'_2}_{J_2}g\|_{L^2}.
\end{split}
\end{displaymath}

Hence, the only thing left to deal with is $I'$. Since now the separation in the second variable is not good enough, we have to use the mixed H\"{o}lder-size condition instead. Again, in the full kernel representation, by cancellation property we can change the kernel to $K(x,y)-K(x,c_{I_1},y_2)$, then
\begin{displaymath}
\begin{split}
|I'|&\lesssim \int_{I_1\times J_1}\int_{I_2\times (3J_1\cap J_{2,1}^c)}\frac{\ell(I_1)^\delta}{|x_1-c_{I_1}|^{n+\delta}}\frac{|\Delta^{b_1}_{I_1}\Delta^{b_2}_{J_1}f(y)|}{|x_2-y_2|^m}(|\Delta^{b'_1}_{I_2}\Delta^{b'_2}_{J_2}g(x)|+|\ave{\Delta^{b'_1}_{I_2}\Delta^{b'_2}_{J_2}g}^{b'_2}_{J_{2,1}}|)\,dxdy\\
&\lesssim 2^{-i_1\delta/2}|K|^{-1}\int_{I_1\times J_1}\int_{I_2\times (3J_1)\cap J_{2,1}^c}\frac{|\Delta^{b_1}_{I_1}\Delta^{b_2}_{J_1}f(y)|}{|x_2-y_2|^m}(|\Delta^{b'_1}_{I_2}\Delta^{b'_2}_{J_2}g(x)|+|\ave{\Delta^{b'_1}_{I_2}\Delta^{b'_2}_{J_2}g}^{b'_2}_{J_{2,1}}|)\,dxdy\\
&\lesssim 2^{-i_1\delta/2}|K|^{-1}\int_{I_1\times J_1}\int_{I_2\times (3J_1\cap J_{2,1}^c)}\frac{|\Delta^{b_1}_{I_1}\Delta^{b_2}_{J_1}f(y)|}{|x_2-y_2|^m}(\sum_{j=2}^{2^m}|\ave{\Delta^{b'_1}_{I_2}\Delta^{b'_2}_{J_2}g}^{b'_2}_{J_{2,j}}|+|\ave{\Delta^{b'_1}_{I_2}\Delta^{b'_2}_{J_2}g}^{b'_2}_{J_{2,1}}|)\,dxdy\\
&\lesssim 2^{-i_1\delta/2}\frac{|I_2|^{1/2}}{|K|}|J_2|^{-1/2}\|\Delta^{b'_1}_{I_2}\Delta^{b'_2}_{J_2}g\|_{L^2}\int_{I_1\times J_1}\int_{3J_1\cap J_{2,1}^c}\frac{|\Delta^{b_1}_{I_1}\Delta^{b_2}_{J_1}f(y)|}{|x_2-y_2|^m}\,dx_2dy\\
&=2^{-i_1\delta/2}\frac{|I_2|^{1/2}}{|K|}|J_2|^{-1/2}\|\Delta^{b'_1}_{I_2}\Delta^{b'_2}_{J_2}g\|_{L^2}\sum_{i=1}^{2^n}\sum_{j=1}^{2^m}\int_{I_{1,i}\times J_{1,j}}\int_{3J_1\cap J_{2,1}^c}\frac{|\ave{\Delta^{b_1}_{I_1}\Delta^{b_2}_{J_1}f(y)}_{I_{1,i}\times J_{1,j}}|}{|x_2-y_2|^m}\,dx_2dy\\
&\lesssim 2^{-i_1\delta/2}\frac{|I_1|^{1/2}|I_2|^{1/2}}{|K|}|J_2|^{-1/2}|J_1|^{-1/2}\|\Delta^{b'_1}_{I_2}\Delta^{b'_2}_{J_2}g\|_{L^2}\|\Delta^{b_1}_{I_1}\Delta^{b_2}_{J_1}f\|_{L^2}\int_{J_1}\int_{3J_1\setminus J_1}\frac{1}{|x_2-y_2|^m}\,dx_2dy_2\\
&\lesssim 2^{-i_1\delta/2}\frac{|I_1|^{1/2}|I_2|^{1/2}}{|K|}\frac{|J_1|^{1/2}}{|J_2|^{1/2}}\|\Delta^{b_1}_{I_1}\Delta^{b_2}_{J_1}f\|_{L^2}\|\Delta^{b'_1}_{I_2}\Delta^{b'_2}_{J_2}g\|_{L^2}\\
&\lesssim 2^{-i_1\delta/2}2^{-j_1\delta/2}\frac{|I_1|^{1/2}|I_2|^{1/2}}{|K|}\frac{|J_1|^{1/2}}{|J_2|^{1/2}}\|\Delta^{b_1}_{I_1}\Delta^{b_2}_{J_1}f\|_{L^2}\|\Delta^{b'_1}_{I_2}\Delta^{b'_2}_{J_2}g\|_{L^2}.
\end{split}
\end{displaymath}
In the above, the fifth line is because $\Delta^{b_1}_{I_1}\Delta^{b_2}_{J_1}f$ is a constant on each child of $I_1\times J_1$, and the last line follows from the fact that the size of $J_1, J_2$ are comparable. This completes the proof of the proposition.
\end{proof}

\subsection*{Part $\sigma_{\out/\inside}''$}
For the part $\sigma_{\out/\inside}''$, we are going to rewrite it into a form containing a partial $b$-adapted paraproduct. Rewrite
\begin{displaymath}
\sigma_{\out/\inside}''=\sum_{i_2=1}^\infty\sum_{i_1=i_2}^\infty\sum_{K}\sum_{\substack{\dist(I_1,I_2)>\ell(I_1)^{\gamma_n}\ell(I_2)^{1-\gamma_n}\\I_1\vee I_2=K}}^{(i_1,i_2)}\sum_{J_1\subsetneq J_2} II,
\end{displaymath}
and first look at the innermost sum.
\begin{displaymath}
\begin{split}
&\sum_{J_1\subsetneq J_2} \pair{M_{b'}TM_b\Delta^{b_1}_{I_1}\Delta^{b_2}_{J_1}f}{\ave{\Delta^{b'_1}_{I_2}\Delta^{b'_2}_{J_2}g}^{b'_2}_{J_{2,1}}(x_1)\otimes 1(x_2)}\\
&=\sum_{J_1\subsetneq J_2}\pair{\pair{M_{b'}TM_b\Delta^{b_1}_{I_1}\Delta^{b_2}_{J_1}f}{1}_2}{\ave{\Delta^{b'_1}_{I_2}\Delta^{b'_2}_{J_2}g}^{b'_2}_{J_1}}_1\\
&=\sum_{J_1}\pair{\pair{M_{b'}TM_b\Delta^{b_1}_{I_1}\Delta^{b_2}_{J_1}f}{1}_2}{\ave{\sum_{J_2\supsetneq J_1}\Delta^{b'_1}_{I_2}\Delta^{b'_2}_{J_2}g}^{b'_2}_{J_1}}_1\\
&=\sum_{J_1}\pair{\pair{M_{b'}TM_b\Delta^{b_1}_{I_1}\Delta^{b_2}_{J_1}f}{1}_2}{\ave{\Delta^{b'_1}_{I_2}g}^{b'_2}_{J_1}}_1\\
&=\sum_V\pair{M_{b'}TM_b\Delta^{b_1}_{I_1}\Delta^{b_2}_Vf}{\ave{\Delta^{b'_1}_{I_2}g}^{b'_2}_V\otimes 1}\\
&=\pair{\Delta^{b_1}_{I_1}f}{\sum_VM_b\Delta^{b_1}_{I_1}\Delta^{b_2}_VT^*(b'_1\ave{\Delta^{b'_1}_{I_2}g}^{b'_2}_V\otimes b'_2)}.
\end{split}
\end{displaymath}

Notice that $\Delta^{b_1}_{I_1}f,\Delta^{b'_1}_{I_2}g$ are constant with respect to $x_1$ on each child of $I_1, I_2$, respectively. If we decompose the above pairing into parts that are restricted on children of $I_1,I_2$, then
\begin{displaymath}
\begin{split}
\sum_{J_1\subsetneq J_2}II&=\sum_{t=1}^{2^n}\sum_{k=1}^{2^m}\pair{\chi_{I_{1,t}}\Delta^{b_1}_{I_1}f}{\sum_V \ave{\Delta^{b'_1}_{I_2}g|_{I_{2,k}}}^{b'_2}_VM_b\Delta^{b_1}_{I_1}\Delta^{b_2}_VT^*(\chi_{I_{2,k}}b'_1\otimes b'_2)}\\
&=\sum_{t=1}^{2^n}\sum_{k=1}^{2^m}\pair{\chi_{I_{1,t}}\Delta^{b_1}_{I_1}f}{b_1\otimes \pi^{b'_2,b_2}_{h_{I_{1,t},I_{2,k}}}(\Delta^{b'_1}_{I_2}g|_{I_{2,k}})},
\end{split}
\end{displaymath}
where $h_{I_{1,t},I_{2,k}}(x_2)=(\Delta^{b_1}_{I_1}T^*(\chi_{I_{2,k}}b'_1\otimes b'_2))|_{I_{1,t}}$, and the following lemma guarantees that the partial paraproduct is well defined.

\begin{lem}\label{hBMO}
$h_{I_{1,t},I_{2,k}}$ is in $BMO(\R^m)$, and satisfies
\begin{displaymath}
\|h_{I_{1,t},I_{2,k}}\|_{BMO}\lesssim 2^{-i_1\delta/2}|K|^{-1}|I_2|.
\end{displaymath}
\end{lem}

We will assume the lemma to be true for the moment and prove it at the end of this section. The above pairing can be further rewritten as:
\begin{displaymath}
\begin{split}
&\sum_{t=1}^{2^n}\sum_{k=1}^{2^m}(\int_{I_{1,t}}b_1\,dx_1)\pair{\Delta^{b_1}_{I_1}f|_{I_{1,t}}}{\pi^{b'_2,b_2}_{h_{I_{1,t},I_{2,k}}}(\Delta^{b'_1}_{I_2}g|_{I_{2,k}})}_2\\
&=\sum_{t=1}^{2^n}\sum_{k=1}^{2^m}(\int_{I_{1,t}}b_1\,dx_1)\pair{\pi^{b'_2,b_2*}_{h_{I_{1,t},I_{2,k}}}(\Delta^{b_1}_{I_1}f|_{I_{1,t}})}{\Delta^{b'_1}_{I_2}g|_{I_{2,k}}}_2\\
&=\sum_{t=1}^{2^n}\sum_{k=1}^{2^m}(\int_{I_{1,t}}b_1\,dx_1)\pair{\frac{\chi_{I_{2,k}}}{|I_{2,k}|}\otimes \pi^{b'_2,b_2*}_{h_{I_{1,t},I_{2,k}}}(\Delta^{b_1}_{I_1}f|_{I_{1,t}})}{\Delta^{b'_1}_{I_2}g}\\
&=\sum_{t=1}^{2^n}\sum_{k=1}^{2^m}(\int_{I_{1,t}}b_1\,dx_1)\pair{\Delta^{b'_1*}_{I_2}(\frac{\chi_{I_{2,k}}}{|I_{2,k}|})\otimes \pi^{b'_2,b_2*}_{h_{I_{1,t},I_{2,k}}}(\Delta^{b_1}_{I_1}f|_{I_{1,t}})}{g}\\
&=\sum_{t=1}^{2^n}\sum_{k=1}^{2^m}(\int_{I_{1,t}}b_1\,dx_1)\pair{b'_1\Delta^{b'_1}_{I_2}(\frac{b^{'-1}_1\chi_{I_{2,k}}}{|I_{2,k}|})\otimes \pi^{b'_2,b_2*}_{h_{I_{1,t},I_{2,k}}}(\Delta^{b_1}_{I_1}f|_{I_{1,t}})}{g}.
\end{split}
\end{displaymath}

Then,
\begin{displaymath}
\begin{split}
&|\sigma_{\out/\inside}''|\\
&=|\sum_{i_2=1}^\infty\sum_{i_1=i_2}^\infty\sum_{K}\sum_{\substack{\dist(I_1,I_2)>\ell(I_1)^{\gamma_n}\ell(I_2)^{1-\gamma_n}\\I_1\vee I_2=K}}^{(i_1,i_2)}\sum_{t=1}^{2^n}\sum_{k=1}^{2^m}(\int_{I_{1,t}}b_1)\pair{b'_1\Delta^{b'_1}_{I_2}(\frac{b^{'-1}_1\chi_{I_{2,k}}}{|I_{2,k}|})\otimes \pi^{b'_2,b_2*}_{h_{I_{1,t},I_{2,k}}}(\Delta^{b_1}_{I_1}f|_{I_{1,t}})}{g}|\\
&\lesssim \sum_{i_2=1}^\infty\sum_{i_1=i_2}^\infty\sum_{t=1}^{2^n}\sum_{k=1}^{2^m}\|\sum_K\sum_{\substack{\dist(I_1,I_2)>\ell(I_1)^{\gamma_n}\ell(I_2)^{1-\gamma_n}\\I_1\vee I_2=K}}^{(i_1,i_2)}(\int_{I_{1,t}}b_1)\Delta^{b'_1}_{I_2}(\frac{b^{'-1}_1\chi_{I_{2,k}}}{|I_{2,k}|})\otimes \pi^{b'_2,b_2*}_{h_{I_{1,t},I_{2,k}}}(\Delta^{b_1}_{I_1}f|_{I_{1,t}})\|_{L^2}\|g\|_{L^2}.
\end{split}
\end{displaymath}

We claim that for any $t,k$,
\begin{equation}\label{parpara}
\|\sum_K\sum_{\substack{\dist(I_1,I_2)>\ell(I_1)^{\gamma_n}\ell(I_2)^{1-\gamma_n}\\I_1\vee I_2=K}}^{(i_1,i_2)}(\int_{I_{1,t}}b_1)\Delta^{b'_1}_{I_2}(\frac{b^{'-1}_1\chi_{I_{2,k}}}{|I_{2,k}|})\otimes \pi^{b'_2,b_2*}_{h_{I_{1,t},I_{2,k}}}(\Delta^{b_1}_{I_1}f|_{I_{1,t}})\|_{L^2}\lesssim 2^{-i_1\delta/2}\|f\|_{L^2}.
\end{equation}

To see this, first observe that since $b'$ is pseudo-accretive, for any $L^2$ function $h$,
\begin{displaymath}
\|h\|_{L^2}\approx \sup_{\|g\|_{L^2}\leq 1}\pair{h}{g}_{b'}=\sup_{\|g\|_{L^2}\leq 1}\int hgb'.
\end{displaymath}
And we have
\begin{displaymath}
\pair{\Delta^{b'_1}_{I_2}h}{g}_{b'}=\pair{h}{\Delta^{b'_1}_{I_2}g}_{b'}.
\end{displaymath}

Hence by linearity, $LHS$ of (\ref{parpara}) is comparable to 
\begin{displaymath}
\begin{split}
&\sup_{\|g\|_{L^2}\leq 1}\sum_K\sum_{\substack{\dist(I_1,I_2)>\ell(I_1)^{\gamma_n}\ell(I_2)^{1-\gamma_n}\\I_1\vee I_2=K}}^{(i_1,i_2)}(\int_{I_{1,t}}b_1)\pair{\Delta^{b'_1}_{I_2}(\frac{b^{'-1}_1\chi_{I_{2,k}}}{|I_{2,k}|})\otimes \pi^{b'_2,b_2*}_{h_{I_{1,t},I_{2,k}}}(\Delta^{b_1}_{I_1}f|_{I_{1,t}})}{g}_{b'}\\
&\lesssim \sup_{\|g\|_{L^2}\leq 1}\sum_K\sum_{I_1,I_1\subset K}^{(i_1,i_2)}|I_1|\|\Delta^{b'_1}_{I_2}(\frac{b^{'-1}_1\chi_{I_{2,k}}}{|I_{2,k}|})\otimes \pi^{b'_2,b_2*}_{h_{I_{1,t},I_{2,k}}}(\Delta^{b_1}_{I_1}f|_{I_{1,t}})\|_{L^2}\|\Delta^{b'_1}_{I_2}g\|_{L^2}.
\end{split}
\end{displaymath}
Since $\|\Delta^{b'_1}_{I_2}(\frac{b^{'-1}_1\chi_{I_{2,k}}}{|I_{2,k}|})\|_{L^2(\R^n)}\lesssim (\int_{I_2}\frac{1}{|I_2|^2})^{1/2}=|I_2|^{-1/2}$, and by Lemma \ref{hBMO}, $\|h_{I_{1,t},I_{2,k}}\|_{BMO}\lesssim 2^{-i_1\delta/2}|K|^{-1}|I_2|$, the $RHS$ of the above inequality
\begin{displaymath}
\begin{split}
&\lesssim 2^{-i_1\delta/2}\sup_{\|g\|_{L^2}\leq 1}\sum_K\sum_{I_1,I_2\subset K}^{(i_1,i_2)}|I_1||I_2|^{1/2}|K|^{-1}\|\Delta^{b_1}_{I_1}f|_{I_{1,t}}\|_{L^2(\R^m)}\|\Delta^{b'_1}_{I_2}g\|_{L^2}\\
&=2^{-i_1\delta/2}\sup_{\|g\|_{L^2}\leq 1}\sum_K\sum_{I_1,I_2\subset K}^{(i_1,i_2)}|I_1||I_2|^{1/2}|K|^{-1}(\frac{1}{|I_{1,t}|}\int_{I_{1,t}}\int_{\R^m}|\Delta^{b_1}_{I_1}f|^2)^{1/2}\|\Delta^{b'_1}_{I_2}g\|_{L^2}\\
&\lesssim 2^{-i_1\delta/2}\sup_{\|g\|_{L^2}\leq 1}\sum_K\sum_{I_1,I_2\subset K}^{(i_1,i_2)}|I_1|^{1/2}|I_2|^{1/2}|K|^{-1}\|\Delta^{b_1}_{I_1}f\|_{L^2}\|\Delta^{b'_1}_{I_2}g\|_{L^2}\\
&\lesssim 2^{-i_1\delta/2}\|f\|_{L^2},
\end{split}
\end{displaymath}
where the last step follows from the first partial control lemma we stated in the beginning.

Then, to complete this section, we give a proof of Lemma \ref{hBMO}.

\begin{proof}(of Lemma \ref{hBMO})
It suffices to show that for any cube $V\subset \R^m$, and any function $a$ satisfying $\mbox{spt}a\subset V,\,|a|\leq 1,\,\int a=0$, there holds
\begin{displaymath}
\pair{h_{I_{1,t},I_{2,k}}}{a}_2\lesssim 2^{-i_1\delta/2}|K|^{-1}|I_2||V|.
\end{displaymath}

To see this, 
\begin{displaymath}
\begin{split}
LHS&=\pair{(\Delta^{b_1}_{I_1}T^*(\chi_{I_{2,k}}b'_1\otimes b'_2))|_{I_{1,t}}}{a}_2\\
&=\pair{\Delta^{b_1}_{I_1}T^*(\chi_{I_{2,k}}b'_1\otimes b'_2)}{\frac{\chi_{I_{1,t}}}{|I_{1,t}|}\otimes a}\\
&=\pair{\chi_{I_{2,k}}b'_1\otimes b'_2}{T(\Delta^{b_1*}_{I_1}(\frac{\chi_{I_{1,t}}}{|I_{1,t}|})\otimes a)}\\
&=\pair{\chi_{I_{2,k}}b'_1\otimes \chi_{3V}b'_2}{T(\Delta^{b_1*}_{I_1}(\frac{\chi_{I_{1,t}}}{|I_{1,t}|})\otimes a)}+\pair{\chi_{I_{2,k}}b'_1\otimes \chi_{(3V)^c}b'_2}{T(\Delta^{b_1*}_{I_1}(\frac{\chi_{I_{1,t}}}{|I_{1,t}|})\otimes a)}\\
&:=(1)+(2).
\end{split}
\end{displaymath}

For $(2)$, since the two functions in the pairing have good separation on both variables, and $\int a=\int \Delta^{b_1*}_{I_1}(\frac{\chi_{I_{1,t}}}{|I_{1,t}|})=0$, use full kernel representation and change the kernel to
\begin{displaymath}
K(x,y)-K(x,y_1,c_V)-K(x,c_{I_1},y_2)+K(x,c_{I_1},c_V).
\end{displaymath}
Then, by H\"{o}lder condition,
\begin{displaymath}
\begin{split}
(2)&\lesssim \ell(I_1)^\delta\ell(V)^\delta \int_{I_1\times V}\int_{I_{2,k}\times (3V)^c}\frac{1}{|x_1-c_{I_1}|^{n+\delta}}\frac{1}{|x_2-c_V|^{m+\delta}}|\Delta^{b_1*}_{I_1}(\frac{\chi_{I_{1,t}}}{|I_{1,t}|})||a|\,dxdy\\
&\lesssim 2^{-i_1\delta/2}|K|^{-1}\ell(V)^\delta |V||I_2|\int_{(3V)^c}\frac{1}{|x_2-c_V|^{m+\delta}}\,dx_2\\
&\lesssim 2^{-i_1\delta/2}|K|^{-1}|V||I_2|.
\end{split}
\end{displaymath}

For $(1)$, there is good separation on only one variable, so we need to use the partial kernel representation.
\begin{displaymath}
\begin{split}
(1)&=\int_{I_1}\int_{I_{2,k}}K_{b_2^{-1}a,\chi_{3V}}(x_1,y_1)\Delta^{b_1*}_{I_1}(\frac{\chi_{I_{1,t}}}{|I_{1,t}|})(y_1)b'_1(x_1)\,dx_1dy_1\\
&=\int_{I_1}\int_{I_{2,k}}(K_{b_2^{-1}a,\chi_{3V}}(x_1,y_1)-K_{b_2^{-1}a,\chi_{3V}}(x_1,c_{I_1}))\Delta^{b_1*}_{I_1}(\frac{\chi_{I_{1,t}}}{|I_{1,t}|})(y_1)b'_1(x_1)\,dx_1dy_1\\
&\lesssim C(b_2^{-1}a,\chi_{3V})(\frac{\ell(I_1)}{\ell(K)})^{\delta/2}|K|^{-1}|I_2|\int_{I_1}|\Delta^{b_1*}_{I_1}(\frac{\chi_{I_{1,t}}}{|I_{1,t}|})|\\
&\lesssim 2^{-i_1\delta/2}|V||K|^{-1}|I_2|.
\end{split}
\end{displaymath}
In the last step of the above, we used the partial C-Z assumption that $C(b_2^{-1}a,\chi_{3V})\lesssim |V|$.
\end{proof}

\section{Separated/Equal: $\sigma_{\out/=}$}
In this part, 
\begin{displaymath}
\sigma_{\out/=}=\sum_{i_2=0}^\infty\sum_{i_1=i_2}^\infty\sum_{K}\sum_{\substack{\dist(I_1,I_2)>\ell(I_1)^{\gamma_n}\ell(I_2)^{1-\gamma_n}\\I_1\vee I_2=K}}^{(i_1,i_2)}\sum_V\pair{M_{b'}TM_b\Delta^{b_1}_{I_1}\Delta^{b_2}_Vf}{\Delta^{b'_1}_{I_2}\Delta^{b'_2}_{V}g}.
\end{displaymath}
By the full control lemma, it suffices to prove the following proposition.
\begin{prop}
\begin{displaymath}
|\pair{M_{b'}TM_b\Delta^{b_1}_{I_1}\Delta^{b_2}_Vf}{\Delta^{b'_1}_{I_2}\Delta^{b'_2}_{V}g}|\lesssim 2^{-i_1\delta/2}\frac{|I_1|^{1/2}|I_2|^{1/2}}{|K|}\|\Delta^{b_1}_{I_1}\Delta^{b_2}_Vf\|_{L^2}\|\Delta^{b'_1}_{I_2}\Delta^{b'_2}_Vg\|_{L^2}.
\end{displaymath}
\end{prop}

\begin{proof}
\begin{displaymath}
\begin{split}
|\pair{M_{b'}TM_b\Delta^{b_1}_{I_1}\Delta^{b_2}_Vf}{\Delta^{b'_1}_{I_2}\Delta^{b'_2}_{V}g}|&\leq \sum_{\substack{V',V''\in\children{(V)}\\V'\neq V''}}|\pair{M_{b'}TM_b(\chi_{V'}\Delta^{b_1}_{I_1}\Delta^{b_2}_Vf)}{\chi_{V''}\Delta^{b'_1}_{I_2}\Delta^{b'_2}_{V}g}|\\
&\qquad +\sum_{V'\in\children{(V)}}|\pair{M_{b'}TM_b(\chi_{V'}\Delta^{b_1}_{I_1}\Delta^{b_2}_Vf)}{\chi_{V'}\Delta^{b'_1}_{I_2}\Delta^{b'_2}_{V}g}|\\
&:=(1)+(2).
\end{split}
\end{displaymath}

For $(2)$, the partial kernel representation gives
\begin{displaymath}
\begin{split}
(2)&=\sum_{V'\in\children{(V)}}|\pair{M_{b'}TM_b(\Delta^{b_1}_{I_1}\Delta^{b_2}_{V}f|_{V'}\otimes \chi_{V'})}{\Delta^{b'_1}_{I_2}\Delta^{b'_2}_Vg|_{V'}\otimes \chi_{V'}}|\\
&=\sum_{V'\in\children{(V)}}|\int_{I_1}\int_{I_2}(K_{\chi_{V'},\chi_{V'}}(x_1,y_1)-K_{\chi_{V'},\chi_{V'}}(x_1,c_{I_1}))\cdot\\
&\qquad \qquad \Delta^{b_1}_{I_1}\Delta^{b_2}_Vf|_{V'}(y_1)\Delta^{b'_1}_{I_2}\Delta^{b'_2}_Vg|_{V'}(x_1)b_1(y_1)b'_1(x_1)\,dx_1dy_1|\\
&\lesssim \sum_{V'\in\children{(V)}}C(\chi_{V'},\chi_{V'})2^{-i_1\delta/2}|K|^{-1}|V'|^{-2}(\int_{I_1\times V'}|\Delta^{b_1}_{I_1}\Delta^{b_2}_Vf|)(\int_{I_2\times V'}|\Delta^{b'_1}_{I_2}\Delta^{b'_2}_Vg|)\\
&\lesssim 2^{-i_1\delta/2}|K|^{-1}|V|^{-1}\|\Delta^{b_1}_{I_1}\Delta^{b_2}_Vf\|_{L^2}\|\Delta^{b'_1}_{I_2}\Delta^{b'_2}_Vg\|_{L^2}|I_1|^{1/2}|I_2|^{1/2}\sum_{V'\in\children{(V)}}|V'|\\
&=2^{-i_1\delta/2}\frac{|I_1|^{1/2}|I_2|^{1/2}}{|K|}\|\Delta^{b_1}_{I_1}\Delta^{b_2}_Vf\|_{L^2}\|\Delta^{b'_1}_{I_2}\Delta^{b'_2}_Vg\|_{L^2}.
\end{split}
\end{displaymath}

For $(1)$, the full kernel representation and the mixed H\"{o}lder-size condition give
\begin{displaymath}
\begin{split}
(1)&=\sum_{\substack{V',V''\in\children{(V)}\\V'\neq V''}}|\int_{I_1\times V'}\int_{I_2\times V''}(K(x,y)-K(x,c_{I_1},y_2))\Delta^{b_1}_{I_1}\Delta^{b_2}_Vf(y_1)\Delta^{b'_1}_{I_2}\Delta^{b'_2}_Vg(x_1)b(y)b'(x)\,dxdy|\\
&\lesssim \sum_{\substack{V',V''\in\children{(V)}\\V'\neq V''}} 2^{-i_1\delta/2}|K|^{-1}(\frac{1}{|V'|}\int_{I_1\times V'}|\Delta^{b_1}_{I_1}\Delta^{b_2}_Vf|)(\frac{1}{|V''|}\int_{I_2\times V''}|\Delta^{b'_1}_{I_2}\Delta^{b'_2}_Vg|)(\int_{V'\times V''}\frac{1}{|x_2-y_2|^m})\\
&\lesssim 2^{-i_1\delta/2}\frac{|I_1|^{1/2}|I_2|^{1/2}}{|K|}\|\Delta^{b_1}_{I_1}\Delta^{b_2}_Vf\|_{L^2}\|\Delta^{b'_1}_{I_2}\Delta^{b'_2}_Vg\|_{L^2}\sum_{\substack{V',V''\in\children{(V)}\\V'\neq V''}}|V'|^{-1/2}|V''|^{-1/2}|V|\\
&\lesssim 2^{-i_1\delta/2}\frac{|I_1|^{1/2}|I_2|^{1/2}}{|K|}\|\Delta^{b_1}_{I_1}\Delta^{b_2}_Vf\|_{L^2}\|\Delta^{b'_1}_{I_2}\Delta^{b'_2}_Vg\|_{L^2},
\end{split}
\end{displaymath}
which completes the proof.
\end{proof}

\section{Separated/Nearby: $\sigma_{\out/\near}$}
In this part, we still want to use the full control lemma to bound the pairing. Notice that since $J_1,J_2$ are near, from a simple lemma proved by Hyt\"{o}nen in \cite{Hy}, the cube $V=J_1\vee J_2$ satisfies $\ell(V)\leq 2^r\ell(J_1)$, hence the size of $J_1$, $J_2$ and $V$ are comparable. Since
\begin{displaymath}
|\sigma_{\out/\near}|\leq \sum_{i_2=1}^\infty\sum_{i_1=i_2}^\infty\sum_{j_1=1}^r\sum_{j_2=1}^{j_1}\sum_{K,V}\sum_{I_1,I_2\subset K}^{(i_1,i_2)}\sum_{J_1,J_2\subset V}^{(j_1,j_2)}|\pair{M_{b'}TM_b\Delta^{b_1}_{I_1}\Delta^{b_2}_{J_1}f}{\Delta^{b'_1}_{I_2}\Delta^{b'_2}_{J_2}g}|
\end{displaymath}
and $\frac{|J_1|^{1/2}|J_2|^{1/2}}{|V|}\approx C$, in order to bound $\sigma_{\out/\near}$, it suffices to show
\begin{displaymath}
|\pair{M_{b'}TM_b\Delta^{b_1}_{I_1}\Delta^{b_2}_{J_1}f}{\Delta^{b'_1}_{I_2}\Delta^{b'_2}_{J_2}g}|\lesssim 2^{-i_1\delta/2}\frac{|I_1|^{1/2}|I_2|^{1/2}}{|K|}\|\Delta^{b_1}_{I_1}\Delta^{b_2}_{J_1}f\|_{L^2}\|\Delta^{b'_1}_{I_2}\Delta^{b'_2}_{J_2}g\|_{L^2}.
\end{displaymath}

To see this, since now both variables are separated but only the first separation is good, by the full kernel representation and the mixed H\"{o}lder-size condition,
\begin{displaymath}
\begin{split}
LHS&=|\int_{I_1\times J_1}\int_{I_2\times J_2}(K(x,y)-K(x,c_{I_1},y_2))\Delta^{b_1}_{I_1}f(y)b(y)\Delta^{b'_1}_{I_2}\Delta^{b'_2}_{J_2}g(x)b'(x)\,dxdy|\\
&\lesssim 2^{-i_1\delta/2}|K|^{-1}\sum_{s,t=1}^{2^m}\int_{I_1\times J_{1,s}}\int_{I_2\times J_{2,t}}|\ave{\Delta^{b_1}_{I_1}\Delta^{b_2}_{J_1}f}_{J_{1,s}}||\ave{\Delta^{b'_1}_{I_2}\Delta^{b'_2}_{J_2}g}_{J_{2,t}}|\frac{1}{|x_2-y_2|^m}\,dxdy\\
&\lesssim 2^{-i_1\delta/2}|K|^{-1}|I_1|^{1/2}|I_2|^{1/2}|J_1|^{-1/2}|J_2|^{-1/2}\|\Delta^{b_1}_{I_1}\Delta^{b_2}_{J_1}f\|_{L^2}\|\Delta^{b'_1}_{I_2}\Delta^{b'_2}_{J_2}g\|_{L^2}\cdot\\
&\qquad\qquad\sum_{s,t=1}^{2^m}\int_{J_{1,s}\times J_{2,t}}\frac{1}{|x_2-y_2|^m}\,dx_2dy_2\\
&\lesssim 2^{-i_1\delta/2}\frac{|I_1|^{1/2}|I_2|^{1/2}}{|K|}\|\Delta^{b_1}_{I_1}\Delta^{b_2}_{J_1}f\|_{L^2}\|\Delta^{b'_1}_{I_2}\Delta^{b'_2}_{J_2}g\|_{L^2},
\end{split}
\end{displaymath}
where the last step follows from the fact that the size of $J_1$, $J_2$ and $V$ are comparable.

\section{Inside/Inside: $\sigma_{\inside/\inside}$}
This part is comparably difficult to deal with, and is also the first place where the assumed $BMO$ conditions stated in the beginning come into play. We will also see that the boundedness of full paraproducts will play an important role in our estimates. To begin with, we first do the following decomposition. Let $I_1\subset I_{2,1}\in\children{(I_2)}, J_1\subset J_{2,1}\in\children{(J_2)}$, then

\begin{displaymath}
\begin{split}
&\pair{M_{b'}TM_b\Delta^{b_1}_{I_1}\Delta^{b_2}_{J_1}f}{\Delta^{b'_1}_{I_2}\Delta^{b'_2}_{J_2}g}\\
&=\pair{M_{b'}TM_b\Delta^{b_1}_{I_1}\Delta^{b_2}_{J_1}f}{\chi_{J_{2,1}^c}(\Delta^{b'_1}_{I_2}\Delta^{b'_2}_{J_2}g-\ave{\Delta^{b'_1}_{I_2}\Delta^{b'_2}_{J_2}g}^{b'_2}_{J_{2,1}})}\\
&\qquad +\pair{M_{b'}TM_b\Delta^{b_1}_{I_1}\Delta^{b_2}_{J_1}f}{\ave{\Delta^{b'_1}_{I_2}\Delta^{b'_2}_{J_2}g}^{b'_2}_{J_{2,1}}(x_1)\otimes 1(x_2)}\\
&=\pair{M_{b'}TM_b\Delta^{b_1}_{I_1}\Delta^{b_2}_{J_1}f}{\chi_{I_{2,1}^c\times J_{2,1}^c}(\Delta^{b'_1}_{I_2}\Delta^{b'_2}_{J_2}g-\ave{\Delta^{b'_1}_{I_2}\Delta^{b'_2}_{J_2}g}^{b'_2}_{J_{2,1}}-\ave{\Delta^{b'_1}_{I_2}\Delta^{b'_2}_{J_2}g}^{b'_1}_{I_{2,1}}+\ave{\Delta^{b'_1}_{I_2}\Delta^{b'_2}_{J_2}g}^{b'}_{I_{2,1}\times J_{2,1}})}\\
&\qquad +\pair{M_{b'}TM_b\Delta^{b_1}_{I_1}\Delta^{b_2}_{J_1}f}{\chi_{J_{2,1}^c}(\ave{\Delta^{b'_1}_{I_2}\Delta^{b'_2}_{J_2}g}^{b'_1}_{I_{2,1}}-\ave{\Delta^{b'_1}_{I_2}\Delta^{b'_2}_{J_2}g}^{b'}_{I_{2,1}\times J_{2,1}})}\\
&\qquad +\pair{M_{b'}TM_b\Delta^{b_1}_{I_1}\Delta^{b_2}_{J_1}f}{\chi_{I_{2,1}^c}(\ave{\Delta^{b'_1}_{I_2}\Delta^{b'_2}_{J_2}g}^{b'_2}_{J_{2,1}}-\ave{\Delta^{b'_1}_{I_2}\Delta^{b'_2}_{J_2}g}^{b'}_{I_{2,1}\times J_{2,1}})}\\
&\qquad +\pair{M_{b'}TM_b\Delta^{b_1}_{I_1}\Delta^{b_2}_{J_1}f}{\ave{\Delta^{b'_1}_{I_2}\Delta^{b'_2}_{J_2}g}^{b'}_{I_{2,1}\times J_{2,1}}1(x_1,x_2)}\\
&:=I+II+III+IV.
\end{split}
\end{displaymath}

\subsection*{Part $II$,$III$}
These two parts are symmetric, so it suffices to estimate one of them, say part $III$. This can be similarly dealt with as the second part in section Separated/Inside, where we used partial paraproducts. 
\begin{displaymath}
\begin{split}
\sum_{J_1\subsetneq J_2}III
&=\sum_{J_1\subsetneq J_2}\pair{\pair{M_{b'}TM_b\Delta^{b_1}_{I_1}\Delta^{b_2}_{J_1}f}{1}_2}{\chi_{I_{2,1}^c}(\ave{\Delta^{b'_1}_{I_2}\Delta^{b'_2}_{J_2}g}^{b'_2}_{J_1}-\ave{\Delta^{b'_1}_{I_2}\Delta^{b'_2}_{J_2}g}^{b'}_{I_1\times J_1})}_1\\
&=\sum_V \pair{\pair{M_{b'}TM_b\Delta^{b_1}_{I_1}\Delta^{b_2}_Vf}{1}_2}{\chi_{I_{2,1}^c}(\ave{\Delta^{b'_1}_{I_2}g}^{b'_2}_V-\ave{\Delta^{b'_1}_{I_2}g}^{b'}_{I_1\times V})}_1\\
&=\sum_V\pair{M_{b'}TM_b\Delta^{b_1}_{I_1}\Delta^{b_2}_Vf}{\chi_{I_{2,1}^c}(\ave{\Delta^{b'_1}_{I_2}g}^{b'_2}_V-\ave{\Delta^{b'_1}_{I_2}g}^{b'}_{I_1\times V})\otimes 1}\\
&=\sum_V\pair{\Delta^{b_1}_{I_1}f}{M_b\Delta^{b_1}_{I_1}\Delta^{b_2}_VT^*(b'_1\chi_{I_{2,1}^c}(\ave{\Delta^{b'_1}_{I_2}g}^{b'_2}_V-\ave{\Delta^{b'_1}_{I_2}g}^{b'}_{I_1\times V})\otimes b'_2)}\\
&=\sum_{k=2}^{2^n}\sum_{t=1}^{2^n}\pair{\chi_{I_{1,t}}\Delta^{b_1}_{I_1}f}{\sum_V\ave{\Delta^{b'_1}_{I_2}g|_{I_{2,k}}}^{b'_2}_VM_b\Delta^{b_1}_{I_1}\Delta^{b_2}_VT^*(\chi_{I_{2,k}}b'_1\otimes b'_2)}\\
&\qquad -\sum_{t=1}^{2^n}\pair{\chi_{I_{1,t}}\Delta^{b_1}_{I_1}f}{\sum_V\ave{\Delta^{b'_1}_{I_2}g|_{I_{2,1}}}^{b'_2}_VM_b\Delta^{b_1}_{I_1}\Delta^{b_2}_VT^*(\chi_{I_{2,1}^c}b'_1\otimes b'_2)}\\
&=\sum_{k=2}^{2^n}\sum_{t=1}^{2^n}\pair{\chi_{I_{1,t}}\Delta^{b_1}_{I_1}f}{b_1\otimes \pi^{b'_2,b_2}_{s_{I_{1,t},I_{2,k}}}(\Delta^{b'_1}_{I_2}g|_{I_{2,k}})}\\
&\qquad -\sum_{t=1}^{2^n}\pair{\chi_{I_{1,t}}\Delta^{b_1}_{I_1}f}{b_1\otimes \pi^{b'_2,b_2}_{s_{I_{1,t},I_{2,1}^c}}(\Delta^{b'_1}_{I_2}g|_{I_{2,1}})},
\end{split}
\end{displaymath} 
where $s_{I_{1,t},I_{2,k}}(x_2)=(\Delta^{b_1}_{I_1}T^*(\chi_{I_{2,k}}b'_1\otimes b'_2))|_{I_{1,t}}$, $s_{I_{1,t},I_{2,1}^c}(x_2)=(\Delta^{b_1}_{I_1}T^*(\chi_{I_{2,1}^c}b'_1\otimes b'_2))|_{I_{1,t}}$. Note that although formally, $s_{I_{1,t},I_{2,k}}$ is exactly the $h_{I_{1,t},I_{2,k}}$ we've encountered in section Separated/Inside, but here since the relative position of $I_1,I_2$ has changed, they are actually different functions. And we will prove later that although $s_{I_{1,t},I_{2,k}}$ is still in $BMO(\R^m)$, the estimate of its norm is different from $h_{I_{1,t},I_{2,k}}$. More specifically,
\begin{lem}\label{sBMO}
\begin{displaymath}
\|s_{I_{1,t},I_{2,k}}\|_{BMO(\R^m)}\lesssim 2^{-i_1\delta/2},\quad \|s_{I_{1,t},I_{2,1}^c}\|_{BMO(\R^m)}\lesssim 2^{-i_1\delta/2}
\end{displaymath}
\end{lem}

Let's assume this to be true right now. Then
\begin{displaymath}
\begin{split}
\sum_{J_1\subsetneq J_2}III&=\sum_{k=2}^{2^n}\sum_{t=1}^{2^n}(\int_{I_{1,t}}b_1)\pair{\Delta^{b_1}_{I_1}f|_{I_{1,t}}}{\pi^{b'_2,b_2}_{s_{I_{1,t},I_{2,k}}}(\Delta^{b'_1}_{I_2}g|_{I_{2,k}})}_2\\
&\qquad -\sum_{t=1}^{2^n}(\int_{I_{1,t}}b_1)\pair{\Delta^{b_1}_{I_1}f|_{I_{1,t}}}{\pi^{b'_2,b_2}_{s_{I_{1,t},I_{2,1}^c}}(\Delta^{b'_1}_{I_2}g|_{I_{2,1}})}_2\\
&=\sum_{k=2}^{2^n}\sum_{t=1}^{2^n}(\int_{I_{1,t}}b_1)\pair{\Delta^{b'_1*}_{I_2}(\frac{\chi_{I_{2,k}}}{|I_{2,k}|})\otimes \pi^{b'_2,b_2*}_{s_{I_{1,t},I_{2,k}}}(\Delta^{b_1}_{I_1}f|_{I_{1,t}})}{g}\\
&\qquad -\sum_{t=1}^{2^n}(\int_{I_{1,t}}b_1)\pair{\Delta^{b'_1*}_{I_2}(\frac{\chi_{I_{2,1}}}{|I_{2,1}|})\otimes \pi^{b'_2,b_2*}_{s_{I_{1,t},I_{2,1}^c}}(\Delta^{b_1}_{I_1}f|_{I_{1,t}})}{g}\\
&:=(1)-(2).
\end{split}
\end{displaymath}

Note that part $(1)$ is exactly the same as the pairing appeared in $\sigma_{\out/\inside}''$, except that here the partial paraproduct is defined using a different $BMO$ function. Hence, following exactly the same argument, for any $t,k$, we have
\begin{displaymath}
\begin{split}
&\|\sum_{I_2}\sum_{I_1\subset I_2}^{(i_1)}(\int_{I_{1,t}}b_1)\Delta^{b'_1*}_{I_2}(\frac{\chi_{I_{2,k}}}{|I_{2,k}|})\otimes\pi^{b'_2,b_2*}_{s_{I_{1,t},I_{2,k}}}(\Delta^{b_1}_{I_1}f|_{I_{1,t}})\|_{L^2}\\
&\lesssim 2^{-i_1\delta/2}\sup_{\|g\|_{L^2}\leq 1}\sum_{I_2}\sum_{I_1\subset I_2}^{(i_1)}\frac{|I_1|^{1/2}}{|I_2|^{1/2}}\|\Delta^{b_1}_{I_1}f\|_{L^2}\|\Delta^{b'_1}_{I_2}g\|_{L^2}\\
&\lesssim 2^{-i_1\delta/2}\|f\|_{L^2},
\end{split}
\end{displaymath}
where again, in the last step, we used the first partial control lemma. 

Similarly, although in part $(2)$, the form of the pairing is a little bit different, however, when dealing with $\Delta^{b'_1*}_{I_2}(\frac{\chi_{I_{2,1}}}{|I_{2,1}|})$, we only need to bound it by 
\begin{displaymath}
C\|\Delta^{b'_1}_{I_2}(\frac{b^{'-1}_1\chi_{I_{2,1}}}{|I_{2,1}|})\|_{L^2(\R^n)}\lesssim |I_2|^{-1/2},
\end{displaymath}
and since the norm of the $BMO$ function has the same bound, so all the rest of the argument for part $(1)$ still works here. i.e. This part satisfies the same estimate as part $(1)$ does.

In conclusion, 

\begin{displaymath}
\begin{split}
|\sum_{I_1\subsetneq I_2}\sum_{J_1\subsetneq J_2}III|&=|\sum_{i_1=1}^\infty\sum_{I_2}\sum_{I_1\subset I_2}^{(i_1)}\sum_{J_1\subsetneq J_2}III|=|\sum_{i_1=1}^\infty\sum_{I_2}\sum_{I_1\subset I_2}^{(i_1)}(1)-(2)|\\
&\lesssim \sum_{i_1=1}^\infty 2^{-i_1\delta/2}\|f\|_{L^2}\|g\|_{L^2}\lesssim \|f\|_{L^2}\|g\|_{L^2}.
\end{split}
\end{displaymath}

And we are only left to prove Lemma \ref{sBMO}:

\begin{proof}(of Lemma \ref{sBMO})
We only prove the inequality for $s_{I_{1,t},I_{2,k}}$, since the other one follows from exactly the same argument. Let cube $V\subset\R^m$, $a$ is any function supported on $V$ such that $|a|\leq 1, \int a=0$. It suffices to show $\pair{s_{I_{1,t},I_{2,k}}}{a}_2\lesssim 2^{-i_1\delta/2}|V|$. 

In the case $\ell(I_1)<2^{-r}\ell(I_2)$, we have $\dist(I_1,I_{2,1}^c)\geq \ell(I_1)^{\gamma_n}\ell(I_2)^{1-\gamma_n}$, i.e. the separation of $I_1$ and $I_{2,k}$ is good enough. Then following from the same reasoning in the proof of Lemma \ref{hBMO}, and note that now $I_2=K$, we have $\pair{s_{I_{1,t},I_{2,k}}}{a}_2\lesssim 2^{-i_1\delta/2}|K|^{-1}|I_2||V|=2^{-i_1\delta/2}|V|$.

Now let's assume $2^{-r}\ell(I_2)\leq\ell(I_1)<\ell(I_2)$. Then the size of $I_1, I_2$ are comparable, i.e. $2^{-i_1}\approx C$, so it suffices to show $\pair{s_{I_{1,t},I_{2,k}}}{a}_2\lesssim |V|$. Split
\begin{displaymath}
\begin{split}
&\pair{s_{I_{1,t},I_{2,k}}}{a}_2\\
&=\pair{\chi_{I_{2,k}}b'_1\otimes b'_2}{T(\Delta^{b_1*}_{I_1}(\frac{\chi_{I_{1,t}}}{|I_{1,t}|})\otimes a)}\\
&=\pair{\chi_{3I_1\cap I_{2,k}}b'_1\otimes \chi_{3V}b'_2}{T(\Delta^{b_1*}_{I_1}(\frac{\chi_{I_{1,t}}}{|I_{1,t}|})\otimes a)}+\pair{\chi_{(3I_1)^c\cap I_{2,k}}b'_1\otimes \chi_{3V}b'_2}{T(\Delta^{b_1*}_{I_1}(\frac{\chi_{I_{1,t}}}{|I_{1,t}|})\otimes a)}\\
&\quad+\pair{\chi_{3I_1\cap I_{2,k}}b'_1\otimes \chi_{(3V)^c}b'_2}{T(\Delta^{b_1*}_{I_1}(\frac{\chi_{I_{1,t}}}{|I_{1,t}|})\otimes a)}+\pair{\chi_{(3I_1)^c\cap I_{2,k}}b'_1\otimes \chi_{(3V)^c}b'_2}{T(\Delta^{b_1*}_{I_1}(\frac{\chi_{I_{1,t}}}{|I_{1,t}|})\otimes a)}\\
&:=(1)+(2)+(3)+(4).
\end{split}
\end{displaymath}

By the partial kernel representation and size condition for the partial kernel,
\begin{displaymath}
\begin{split}
(1)&=\int_{I_1}\int_{3I_1\cap I_{2,k}}K_{b^{-1}_2a,\chi_{3V}}(x_1,y_1)\Delta^{b_1*}_{I_1}(\frac{\chi_{I_{1,t}}}{|I_{1,t}|})(y_1)b'_1(x_1)\,dx_1dy_1\\
&\lesssim C(b^{-1}_2a,\chi_{3V})|I_1|^{-1}\int_{I_1}\int_{3I_1\cap I_{2,k}}\frac{1}{|x_1-y_1|^n}\,dx_1dy_1\lesssim  |V|.
\end{split}
\end{displaymath}

By the partial kernel representation and H\"{o}lder condition for the partial kernel,
\begin{displaymath}
\begin{split}
(2)&=\int_{I_1}\int_{(3I_1)^c\cap I_{2,k}}(K_{b^{-1}_2a,\chi_{3V}}(x_1,y_1)-K_{b^{-1}_2a,\chi_{3V}}(x_1,c_{I_1}))\Delta^{b_1*}_{I_1}(\frac{\chi_{I_{1,t}}}{|I_{1,t}|})(y_1)b'_1(x_1)\,dx_1dy_1\\
&\lesssim C(b^{-1}_2a,\chi_{3V})|I_1|^{-1}\int_{I_1}\int_{(3I_1)^c\cap I_{2,k}} \frac{\ell(I_1)^\delta}{|x_1-c_{I_1}|^{n+\delta}}\lesssim |V|.
\end{split}
\end{displaymath}

By the full kernel representation and mixed H\"{o}lder-size condition,
\begin{displaymath}
\begin{split}
(3)&=\int_{I_1\times V}\int_{3I_1\cap I_{2,k}\times (3V)^c}(K(x,y)-K(x,y_1,c_V))b'(x)\Delta^{b_1*}_{I_1}(\frac{\chi_{I_{1,t}}}{|I_{1,t}|})(y_1)a(y_2)\,dxdy\\
&\lesssim |I_1|^{-1}\ell(V)^\delta\int_{I_1\times V}\int_{3I_1\cap I_{2,k}\times (3V)^c}\frac{1}{|x_1-y_1|^n}\frac{1}{|x_2-c_V|^{m+\delta}}\,dxdy\\
&\lesssim |I_1|^{-1}\ell(V)^\delta |V||I_1|\ell(V)^{-\delta}=|V|.
\end{split}
\end{displaymath}

By the full kernel representation and H\"{o}lder condition,
\begin{displaymath}
\begin{split}
(4)&=\int_{I_1\times V}\int_{(3I_1)^c\cap I_{2,k}\times (3V)^c}(K(x,y)-K(x,c_{I_1},y_2)-K(x,y_1,c_V)+K(x,c_{I_1},c_V))\cdot\\
&\qquad b'(x)\Delta^{b_1*}_{I_1}(\frac{\chi_{I_{1,t}}}{|I_{1,t}|})(y_1)a(y_2)\,dxdy\\
&\lesssim \ell(I_1)^\delta\ell(V)^\delta|I_1|^{-1}\int_{I_1\times V}\int_{(3I_1)^c\cap I_{2,k}\times (3V)^c}\frac{1}{|x_1-c_{I_1}|^{n+\delta}}\frac{1}{|x_2-c_V|^{m+\delta}}\,dxdy\\
&\lesssim |I_2|^{-1}\ell(V)^\delta|I_1|^{-1}|I_1||V||I_2|\ell(V)^{-\delta}=|V|.
\end{split}
\end{displaymath}

Hence, the proof is complete.
\end{proof}

\subsection*{Part $I$}
In part $I$, since the functions in the pairing are separated on both variables, by an argument similar to what we did in the section Separated/Inside, 
\begin{displaymath}
|I|\lesssim (\frac{\ell(I_1)}{\ell(I_2)})^{\delta/2}(\frac{\ell(J_1)}{\ell(J_2)})^{\delta/2}(\frac{|I_1|}{|I_2|})^{1/2}(\frac{|J_1|}{|J_2|})^{1/2}\|\Delta^{b_1}_{I_1}\Delta^{b_2}_{J_1}f\|_{L^2}\|\Delta^{b'_1}_{I_2}\Delta^{b'_2}_{J_2}g\|_{L^2},
\end{displaymath} 
which combined with the full control lemma, will give the boundedness of part $I$. (Note that in order to prove the above inequality, we need to discuss four different cases depending on whether $\ell(I_1)<2^{-r}\ell(I_2)$ and whether $\ell(J_1)<2^{-r}\ell(J_2)$, and use size, H\"{o}lder, or mixed H\"{o}lder-size conditions accordingly in each case.)

\subsection*{Part $IV$}
To deal with this part, we need to use the $b$-adapted full paraproducts and its $L^2\rightarrow L^2$ boundedness. Write
\begin{displaymath}
\begin{split}
&\sum_{I_1\subsetneq I_2}\sum_{J_1\subsetneq J_2}\pair{M_{b'}TM_b\Delta^{b_1}_{I_1}\Delta^{b_2}_{J_1}f}{\ave{\Delta^{b'_1}_{I_2}\Delta^{b'_2}_{J_2}g}^{b'}_{I_1\times J_1}}\\
&=\sum_{J_1\subsetneq J_2}\sum_{I_1}\pair{M_{b'}TM_b\Delta^{b_1}_{I_1}\Delta^{b_2}_{J_1}f}{\ave{\sum_{I_2\supsetneq I_1}\Delta^{b'_1}_{I_2}\Delta^{b'_2}_{J_2}g}^{b'}_{I_1\times J_1}}\\
&=\sum_{J_1\subsetneq J_2}\sum_{I_1}\pair{M_{b'}TM_b\Delta^{b_1}_{I_1}\Delta^{b_2}_{J_1}f}{\ave{\Delta^{b'_2}_{J_2}g}^{b'}_{I_1\times J_1}}\\
&=\sum_{I_1}\sum_{J_1}\pair{M_{b'}TM_b\Delta^{b_1}_{I_1}\Delta^{b_2}_{J_1}f}{\ave{g}^{b'}_{I_1\times J_1}}\\
&=\pair{f}{\pi^{b',b}_{T^*b'}(g)}.
\end{split}
\end{displaymath}

By assumption, $T^*b'\in BMO(\R^n\times\R^m)$, then use the $L^2$ boundedness of the full paraproduct, we have
\begin{displaymath}
\sum_{I_1\subsetneq I_2}\sum_{J_1\subsetneq J_2}IV\lesssim \|T^*b'\|_{BMO}\|f\|_{L^2}\|g\|_{L^2}.
\end{displaymath}

\section{Inside/Equal and Inside/Nearby: $\sigma_{\inside/=},\sigma_{\inside/\near}$}
The ways to estimate these two parts are similar, so we only explain the first one as an example. Let $I_1\subset I_{2,1}\in\children{(I_2)}$, split
\begin{displaymath}
\begin{split}
\sigma_{\inside/=}&=\sum_{I_1\subsetneq I_2}\sum_V\pair{M_{b'}TM_b\Delta^{b_1}_{I_1}\Delta^{b_2}_Vf}{\Delta^{b'_1}_{I_2}\Delta^{b'_2}_Vg}\\
&=\sum_{I_1\subsetneq I_2}\sum_V\pair{M_{b'}TM_b\Delta^{b_1}_{I_1}\Delta^{b_2}_Vf}{\chi_{I_{2,1}^c}(\Delta^{b'_1}_{I_2}\Delta^{b'_2}_Vg-\ave{\Delta^{b'_1}_{I_2}\Delta^{b'_2}_Vg}^{b'_1}_{I_{2,1}})}\\
&\qquad +\sum_{I_1\subsetneq I_2}\sum_V\pair{M_{b'}TM_b\Delta^{b_1}_{I_1}\Delta^{b_2}_Vf}{1(x_1)\otimes\ave{\Delta^{b'_1}_{I_2}\Delta^{b'_2}_{V}g}^{b'_1}_{I_1}(x_2)}\\
&:=\sigma_{\inside/=}'+\sigma_{\inside/=}''.
\end{split}
\end{displaymath}

To bound $\sigma_{\inside/=}'$. In the case $\ell(I_1)<2^{-r}\ell(I_2)$, it can be dealt with similarly as in the case Separated/Equal. In the case $2^{-r}\ell(I_2)\leq\ell(I_1)<\ell(I_2)$, we claim that 
\begin{displaymath}
|\sigma_{\inside/=}'|\lesssim \sum_{i_1=1}^\infty\sum_{I_2}\sum_{I_1\subset I_2}^{(i_1)}\sum_V2^{-i_1\delta/2}\frac{|I_1|^{1/2}}{|I_2|^{1/2}}\|\Delta^{b_1}_{I_1}\Delta^{b_2}_{V}f\|_{L^2}\|\Delta^{b'_1}_{I_2}\Delta^{b'_2}_Vg\|_{L^2},
\end{displaymath}
then the full control lemma implies the correct bound.

In order to prove the claim, further split
\begin{displaymath}
\begin{split}
&\pair{M_{b'}TM_b\Delta^{b_1}_{I_1}\Delta^{b_2}_Vf}{\chi_{I_{2,1}^c}(\Delta^{b'_1}_{I_2}\Delta^{b'_2}_Vg-\ave{\Delta^{b'_1}_{I_2}\Delta^{b'_2}_Vg}^{b'_1}_{I_{2,1}})}\\
&=\sum_{\substack{V',V''\in\children{(V)}\\V'\neq V''}}\pair{M_{b'}TM_b(\chi_{V'}\Delta^{b_1}_{I_1}\Delta^{b_2}_Vf)}{\chi_{3I_1\cap I_{2,1}^c}\otimes\chi_{V''}(\Delta^{b'_1}_{I_2}\Delta^{b'_2}_Vg-\ave{\Delta^{b'_1}_{I_2}\Delta^{b'_2}_Vg}^{b'_1}_{I_{2,1}})}\\
&\quad +\sum_{\substack{V',V''\in\children{(V)}\\V'\neq V''}}\pair{M_{b'}TM_b(\chi_{V'}\Delta^{b_1}_{I_1}\Delta^{b_2}_Vf)}{\chi_{(3I_1)^c\cap I_{2,1}^c}\otimes\chi_{V''}(\Delta^{b'_1}_{I_2}\Delta^{b'_2}_Vg-\ave{\Delta^{b'_1}_{I_2}\Delta^{b'_2}_Vg}^{b'_1}_{I_{2,1}})}\\
&\quad +\sum_{V'\in\children{(V)}}\pair{M_{b'}TM_b(\chi_{V'}\Delta^{b_1}_{I_1}\Delta^{b_2}_Vf)}{\chi_{3I_1\cap I_{2,1}^c}\otimes\chi_{V'}(\Delta^{b'_1}_{I_2}\Delta^{b'_2}_Vg-\ave{\Delta^{b'_1}_{I_2}\Delta^{b'_2}_Vg}^{b'_1}_{I_{2,1}})}\\
&\quad +\sum_{V'\in\children{(V)}}\pair{M_{b'}TM_b(\chi_{V'}\Delta^{b_1}_{I_1}\Delta^{b_2}_Vf)}{\chi_{(3I_1)^c\cap I_{2,1}^c}\otimes\chi_{V'}(\Delta^{b'_1}_{I_2}\Delta^{b'_2}_Vg-\ave{\Delta^{b'_1}_{I_2}\Delta^{b'_2}_Vg}^{b'_1}_{I_{2,1}})}\\
&:=(1)+(2)+(3)+(4).
\end{split}
\end{displaymath}

In part $(1)$ and $(2)$, both variables are separated, so we use the full kernel representation. And by the size condition and the mixed H\"{o}lder-size condition, respectively, they are bounded. In part $(3)$ and $(4)$, only the first variable is separated, so we need the partial kernel representation. By the size condition and H\"{o}lder condition for the partial kernel, respectively, they are bounded as well. We omit the details.

Now we deal with $\sigma_{\inside/=}''$, which needs the partial paraproduct argument, but is much easier than the cases we've seen before. As before, rewrite
\begin{displaymath}
\begin{split}
&\sum_{I_1\subsetneq I_2}\pair{M_{b'}TM_b\Delta^{b_1}_{I_1}\Delta^{b_2}_Vf}{\ave{\Delta^{b'_1}_{I_2}\Delta^{b'_2}_Vg}^{b'_1}_{I_1}}\\
&=\sum_{K}\pair{M_{b'}TM_b\Delta^{b_1}_K\Delta^{b_2}_Vf}{1(x_1)\otimes\ave{\Delta^{b'_2}_Vg}^{b'_1}_K(x_2)}\\
&=\sum_{t=1}^{2^m}\sum_{k=1}^{2^m}\pair{\chi_{V_t}\Delta^{b_2}_Vf}{\pi^{b'_1,b_1}_{r_{V_t,V_k}}(\Delta^{b'_2}_Vg|_{V_k})\otimes b_2}\\
&=\sum_{t=1}^{2^m}\sum_{k=1}^{2^m}(\int_{V_t}b_2)\pair{\pi^{b'_1,b_1*}_{r_{V_t,V_k}}(\Delta^{b_2}_Vf|_{V_t})\otimes\Delta^{b'_2*}_V(\frac{\chi_{V_k}}{|V_k|})}{g},
\end{split}
\end{displaymath}
where $r_{V_t,V_k}(x_1)=(\Delta^{b_2}_VT^*(b'_1\otimes \chi_{V_k}b'_2))|_{V_t}$ is a $BMO$ function whose norm satisfies the following lemma.

\begin{lem}\label{rBMO}
\begin{displaymath}
\|r_{V_t,V_k}\|_{BMO(\R^n)}\lesssim C.
\end{displaymath}
\end{lem}

We postpone the proof, and assume this bound for the moment. Then
\begin{displaymath}
|\sigma_{\inside/=}''|\leq\sum_{t=1}^{2^m}\sum_{k=1}^{2^m}\|\sum_V (\int_{V_t}b_2)\pi^{b'_1,b_1*}_{r_{V_t,V_k}}(\Delta^{b_2}_Vf|_{V_t})\otimes\Delta^{b'_2*}_V(\frac{\chi_{V_k}}{|V_k|})\|_{L^2}\|g\|_{L^2}.
\end{displaymath}

By a similar argument as in the previous two partial paraproducts, involving the estimate of the $BMO$ norm of $r_{V_t,V_k}$ and the $L^2$ boundedness of the partial paraproduct, it is not hard to show that for any $t,k$,
\begin{displaymath}
\|\sum_V |\int_{V_t}b_2|\pi^{b'_1,b_1*}_{r_{V_t,V_k}}(\Delta^{b_2}_Vf|_{V_t})\otimes\Delta^{b'_2*}_V(\frac{\chi_{V_k}}{|V_k|})\|_{L^2}\lesssim \|f\|_{L^2},
\end{displaymath}
which completes the estimate of part $\sigma_{\inside/=}''$.

\begin{proof}(of Lemma \ref{rBMO})
For any cube $K\subset\R^n$ and any function $a$ supported on $K$ such that $|a|\leq 1, \int a=0$, we claim that $\pair{r_{V_t,V_k}}{a}_1\lesssim |K|$.

To see this, write
\begin{displaymath}
\begin{split}
\pair{r_{V_t,V_k}}{a}_1&=\pair{\chi_Kb'_1\otimes\chi_{V_k}b'_2}{T(a\otimes\Delta^{b_2*}_V(\frac{\chi_{V_t}}{|V_t|}))}\\
&\qquad +\pair{\chi_{(3K)^c}b'_1\otimes\chi_{V_k}b'_2}{T(a\otimes\Delta^{b_2*}_V(\frac{\chi_{V_t}}{|V_t|}))}\\
&\qquad +\pair{\chi_{3K\setminus K}b'_1\otimes\chi_{V_k}b'_2}{T(a\otimes\Delta^{b_2*}_V(\frac{\chi_{V_t}}{|V_t|}))}\\
&:=(1)+(2)+(3).
\end{split}
\end{displaymath}

For part $(1)$, write
\begin{displaymath}
(1)=\sum_{s=1}^{2^m}\pair{\chi_Kb'_1\otimes\chi_{V_k}b'_2}{T(a\otimes\chi_{V_s}\Delta^{b_2*}_V(\frac{\chi_{V_t}}{|V_t|}))}.
\end{displaymath}
If $s\neq k$, use partial kernel representation and size condition for the partial kernel,
\begin{displaymath}
\begin{split}
&\pair{\chi_Kb'_1\otimes\chi_{V_k}b'_2}{T(a\otimes\chi_{V_s}\Delta^{b_2*}_V(\frac{\chi_{V_t}}{|V_t|}))}\\
&=\int_{V_s}\int_{V_k}K_{b_1^{-1}a,\chi_K}(x_2,y_2)b'_2(x_2)\Delta^{b_2*}_V(\frac{\chi_{V_t}}{|V_t|})(y_2)\,dx_2dy_2\\
&\lesssim C(b_1^{-1}a,\chi_{K})|V|^{-1}\int_{V_s}\int_{V_k}\frac{1}{|x_2-y_2|^m}\lesssim |K|.
\end{split}
\end{displaymath}
If $s=k$, by the first diagonal $BMO$ condition, 
\begin{displaymath}
\begin{split}
&\pair{\chi_Kb'_1\otimes\chi_{V_k}b'_2}{T(a\otimes\chi_{V_k}\Delta^{b_2*}_V(\frac{\chi_{V_t}}{|V_t|}))}\\
&=(\Delta^{b_2}_V(\frac{b^{-1}_2\chi_{V_t}}{|V_t|})|_{V_k})\pair{\chi_Kb'_1\otimes\chi_{V_k}b'_2}{T(a\otimes\chi_{V_k}b_2)}\\
&\lesssim |V|^{-1}|K||V|=|K|.
\end{split}
\end{displaymath}

For part $(2)$ and $(3)$, write
\begin{displaymath}
(2)=\sum_{s=1}^{2^m}\pair{\chi_{(3K)^c}b'_1\otimes\chi_{V_k}b'_2}{T(a\otimes\chi_{V_s}\Delta^{b_2*}_V(\frac{\chi_{V_t}}{|V_t|}))},
\end{displaymath}
and similarly for $(3)$.

If $s\neq k$, since both variables are separated,  we can use full kernel representation, and mixed H\"{o}lder-size condition for $(2)$, size condition for $(3)$. If $s=k$, we use partial kernel representation, and H\"{o}lder condition for $(2)$, size condition for $(3)$. The details can be carried out similarly as for $(1)$, and we omit them.
\end{proof}

\section{Equal/Equal, Equal/Nearby and Nearby/Nearby: $\sigma_{=/=}$}
We discuss these three cases together. When $J_1, J_2$ are near each other, the sizes of $J_1, J_2, J_1\vee J_2$ are comparable, similarly for the other variable. So by the full control lemma, in either of these three cases, it suffices to show 
\begin{displaymath}
|\pair{M_{b'}TM_b\Delta^{b_1}_{I_1}\Delta^{b_2}_{J_1}f}{\Delta^{b'_1}_{I_2}\Delta^{b'_2}_{J_2}g}|\lesssim \|\Delta^{b_1}_{I_1}\Delta^{b_2}_{J_1}f\|_{L^2}\|\Delta^{b'_1}_{I_2}\Delta^{b'_2}_{J_2}g\|_{L^2}.
\end{displaymath}

We only prove the above for the case Equal/Equal, which is the most difficult one since there is no separation on either variable. Note that for Equal/Nearby, one can use partial kernel representation and size condition to prove it, and for Nearby/Nearby, the full kernel representation and size condition will do.

Write $I_1=I_2=K, J_1=J_2=V$, and decompose the pairing into restrictions on each pair of their children, 
\begin{displaymath}
|\pair{M_{b'}TM_b\Delta^{b_1}_K\Delta^{b_2}_Vf}{\Delta^{b'_1}_K\Delta^{b'_2}_Vg}|\leq \sum_{i,s=1}^{2^n}\sum_{j,t=1}^{2^m}|\pair{M_{b'}TM_b(\chi_{K_i\times V_j}\Delta^{b_1}_K\Delta^{b_2}_Vf)}{\chi_{K_s\times V_t}\Delta^{b'_1}_K\Delta^{b'_2}_Vg}|.
\end{displaymath}

If $i\neq s, j\neq t$, by the full kernel representation and size condition,
\begin{displaymath}
\begin{split}
&|\pair{M_{b'}TM_b(\chi_{K_i\times V_j}\Delta^{b_1}_K\Delta^{b_2}_Vf)}{\chi_{K_s\times V_t}\Delta^{b'_1}_K\Delta^{b'_2}_Vg}|\\
&\lesssim |\ave{\Delta^{b_1}_K\Delta^{b_2}_Vf}_{K_i\times V_j}||\ave{\Delta^{b'_1}_K\Delta^{b'_2}_Vg}_{K_s\times V_t}|\int_{K_i\times V_j}\int_{K_s\times V_t}\frac{1}{|x_1-y_1|^n}\frac{1}{|x_2-y_2|^m}\\
&\lesssim \|\Delta^{b_1}_K\Delta^{b_2}_Vf\|_{L^2}|K|^{-1/2}|V|^{-1/2}\|\Delta^{b'_1}_K\Delta^{b'_2}_Vg\|_{L^2}|K|^{-1/2}|V|^{-1/2}|K||V|=\|\Delta^{b_1}_K\Delta^{b_2}_Vf\|_{L^2}\|\Delta^{b'_1}_K\Delta^{b'_2}_Vg\|_{L^2}.
\end{split}
\end{displaymath}

If $i\neq s, j=t$, by the partial kernel representation and size condition for the partial kernel,
\begin{displaymath}
\begin{split}
&|\pair{M_{b'}TM_b(\chi_{K_i\times V_j}\Delta^{b_1}_K\Delta^{b_2}_Vf)}{\chi_{K_s\times V_j}\Delta^{b'_1}_K\Delta^{b'_2}_Vg}|\\
&\lesssim |\ave{\Delta^{b_1}_K\Delta^{b_2}_Vf}_{K_i\times V_j}||\ave{\Delta^{b'_1}_K\Delta^{b'_2}_Vg}_{K_s\times V_j}|\int_{K_i}\int_{K_s}|K_{\chi_{V_j},\chi_{V_j}}(x_1,y_1)|\\
&\lesssim \|\Delta^{b_1}_K\Delta^{b_2}_Vf\|_{L^2}|K|^{-1/2}|V|^{-1/2}\|\Delta^{b'_1}_K\Delta^{b'_2}_Vg\|_{L^2}|K|^{-1/2}|V|^{-1/2}C(\chi_{V_j},\chi_{V_j})|K|\\
&\lesssim\|\Delta^{b_1}_K\Delta^{b_2}_Vf\|_{L^2}\|\Delta^{b'_1}_K\Delta^{b'_2}_Vg\|_{L^2}.
\end{split}
\end{displaymath}
The case $i=s, j\neq t$ is symmetric to this one.

If $i=s, j=t$, by the weak boundedness property,
\begin{displaymath}
\begin{split}
&|\pair{M_{b'}TM_b(\chi_{K_i\times V_j}\Delta^{b_1}_K\Delta^{b_2}_Vf)}{\chi_{K_i\times V_j}\Delta^{b'_1}_K\Delta^{b'_2}_Vg}|\\
&=|\ave{\Delta^{b_1}_K\Delta^{b_2}_Vf}_{K_i\times V_j}||\ave{\Delta^{b'_1}_K\Delta^{b'_2}_Vg}_{K_i\times V_j}||\pair{M_{b'}TM_b(\chi_{K_i}\otimes\chi_{V_j})}{\chi_{K_i}\otimes\chi_{V_j}}|\\
&\lesssim \|\Delta^{b_1}_K\Delta^{b_2}_Vf\|_{L^2}|K|^{-1/2}|V|^{-1/2}\|\Delta^{b'_1}_K\Delta^{b'_2}_Vg\|_{L^2}|K|^{-1/2}|V|^{-1/2} |K_i||V_j|\\
&\lesssim \|\Delta^{b_1}_K\Delta^{b_2}_Vf\|_{L^2}\|\Delta^{b'_1}_K\Delta^{b'_2}_Vg\|_{L^2}
\end{split}
\end{displaymath}

This completes this section, as well as all the cases when $\ell(I_1)\leq\ell(I_2), \ell(J_1)\leq\ell(J_2)$. Moreover, the cases when $\ell(I_1)>\ell(I_2), \ell(J_1)>\ell(J_2)$ can be dealt with symmetrically. 

\section{Mixed cases}

We now consider the mixed cases. It suffices to analyze the case when $\ell(I_1)\leq\ell(I_2), \ell(J_1)>\ell(J_2)$, and the only sub-case which is not symmetric to any of the above is the mixed Inside/Inside, which involves the boundedness of mixed paraproducts. By assumption, $I_1\subset I_2, J_2\subsetneq J_1$. Suppose $I_1\subset I_{2,1}\in\children{(I_2)}$ and $J_2\subset J_{1,2}\in\children{(J_1)}$. Split
\begin{displaymath}
\begin{split}
&\pair{M_{b'}TM_b\Delta^{b_1}_{I_1}\Delta^{b_2}_{J_1}f}{\Delta^{b'_1}_{I_2}\Delta^{b'_2}_{J_2}g}\\
&=\pair{M_{b'}TM_b\Delta^{b_1}_{I_1}\Delta^{b_2}_{J_1}f}{\chi_{I_{2,1}^c}(\Delta^{b'_1}_{I_2}\Delta^{b'_2}_{J_2}g-\ave{\Delta^{b'_1}_{I_2}\Delta^{b'_2}_{J_2}g}^{b'_1}_{I_{2,1}})}\\
&\qquad +\pair{M_{b'}TM_b\Delta^{b_1}_{I_1}\Delta^{b_2}_{J_1}f}{1(x_1)\otimes\ave{\Delta^{b'_1}_{I_2}\Delta^{b'_2}_{J_2}g}^{b'_1}_{I_{2,1}}(x_2)}\\
&=\pair{M_{b'}TM_b(\chi_{J_{1,2}^c}(\Delta^{b_1}_{I_1}\Delta^{b_2}_{J_1}f-\ave{\Delta^{b_1}_{I_1}\Delta^{b_2}_{J_1}f}^{b_2}_{J_{1,2}}))}{\chi_{I_{2,1}^c}(\Delta^{b'_1}_{I_2}\Delta^{b'_2}_{J_2}g-\ave{\Delta^{b'_1}_{I_2}\Delta^{b'_2}_{J_2}g}^{b'_1}_{I_{2,1}})}\\
&\qquad +\pair{M_{b'}TM_b\ave{\Delta^{b_1}_{I_1}\Delta^{b_2}_{J_1}f}^{b_2}_{J_{1,2}}}{\chi_{I_{2,1}^c}(\Delta^{b'_1}_{I_2}\Delta^{b'_2}_{J_2}g-\ave{\Delta^{b'_1}_{I_2}\Delta^{b'_2}_{J_2}g}^{b'_1}_{I_{2,1}})}\\
&\qquad +\pair{\chi_{J_{1,2}^c}(\delta^{b_1}_{I_1}\Delta^{b_2}_{J_1}f-\ave{\Delta^{b_1}_{I_1}\Delta^{b_2}_{J_1}f}^{b_2}_{J_{1,2}})}{M_bT^*M_{b'}\ave{\Delta^{b'_1}_{I_2}\Delta^{b'_2}_{J_2}g}^{b'_1}_{I_{2,1}}}\\
&\qquad +\pair{M_{b'}TM_b\ave{\Delta^{b_1}_{I_1}\Delta^{b_2}_{J_1}f}^{b_2}_{J_{1,2}}}{\ave{\Delta^{b'_1}_{I_2}\Delta^{b'_2}_{J_2}g}^{b'_1}_{I_{2,1}}}\\
&:=I+II+III+IV.
\end{split}
\end{displaymath}

Part $I,II,III$ can be similarly estimated as the corresponding parts in the Inside/Inside case discussed above. Note that for part $II,III$, we need to use the partial adjoint operator $T_1$ to rewrite it into a form having partial paraproduct in it, and estimate some new one-parameter $BMO$ functions, which can be achieved by the same techniques we've seen before.

To estimate part $IV$, we need to apply the boundedness of mixed paraproducts.
\begin{displaymath}
\begin{split}
\sum_{I_1\subsetneq I_2}\sum_{J_2\subsetneq J_1}IV&=\sum_{K,V}\pair{M_{b'}TM_b\ave{\Delta^{b_1}_Kf}^{b_2}_V}{\ave{\Delta^{b'_2}_Vg}^{b'_1}_K}\\
&=\sum_{K,V}\pair{T(b_1\ave{\Delta^{b_1}_Kf}^{b_2}_V\otimes b_2)}{b'_1\otimes b'_2\ave{\Delta^{b'_2}_Vg}^{b'_1}_K}\\
&=\sum_{K,V}\pair{T_1(b'_1\otimes b_2)}{b_1\ave{\Delta^{b_1}_Kf}^{b_2}_V\otimes b'_2\ave{\Delta^{b'_2}_Vg}^{b'_1}_K}.
\end{split}
\end{displaymath}

Recall that by assumption, $d=b_1\otimes b'_2, d'=b'_1\otimes b_2$, so the above is
\begin{displaymath}
\begin{split}
&\pair{T_1(d')}{\sum_{K,V}M_d\ave{\Delta^{d_1}_Kf}^{d'_2}_V\otimes \ave{\Delta^{d_2}_Vg}^{d'_1}_K}\\
&=\sum_{K,V}\pair{T_1(d')}{M_d(E^{d'_2}_V\Delta^{d_1}_Kf)(E^{d'_1}_K\Delta^{d_2}_Vg)}\\
&=\pair{\sum_{K,V}E^{d'_1*}_K((E^{d'_2}_Vf)M_d\Delta^{d_1}_K\Delta^{d_2}_VT_1(d'))}{g}\\
&=\pair{\tilde{\pi}^{d',d}_{T_1(d')}(f)}{g}\\
&\lesssim \|T_1(d')\|_{BMO}\|f\|_{L^2}\|g\|_{L^2},
\end{split}
\end{displaymath}
and $\|T_1(d')\|_{BMO}<\infty$ is one of our $BMO$ assumptions. This completes the estimate of the mixed cases.

\end{document}